%% file: agp.tex
\documentclass[
 hidempi, 
]{mpi2015-cscpreprint}

\input{ex_shared}
\input{ex_math}

\usepackage{amsthm}
\usepackage{algorithm2e}

\newtheorem{theorem}{Theorem}[section]
\newtheorem{definition}{Definition}[section]
\newtheorem{corollary}{Corollary}[section]
\newtheorem{corr}{Corollary}[section]
\newtheorem{remark}{Remark}[section]
\newtheorem{result}{Result}[section]
\newtheorem{proposition}{Proposition}[section]
\newtheorem{example}{Example}[section]

\begin{document}

\maketitle

%%%%%%%%%%%%%%%%%%%%%%%%%%%%%%%%%
\section{Introduction}
\input{sec-intro}

%%%%%%%%%%%%%%%%%%%%%%%%%%%%%%%%%
\section{Realizations of multivariate rational functions} \label{sec:realization}
\input{sec-realization}

%%%%%%%%%%%%%%%%%%%%%%%%%%%%%%%%%
\section{Data definitions and description}\label{sec:data}
\input{sec-data}

%%%%%%%%%%%%%%%%%%%%%%%%%%%%%%%%%
\section{Multivariate Loewner matrices and null spaces}\label{sec:LL}
\input{sec-LL}

%%%%%%%%%%%%%%%%%%%%%%%%%%%%%%%%%
\section{Addressing the curse of dimensionality}\label{sec:cod}
\input{sec-LLcod}

%%%%%%%%%%%%%%%%%%%%%%%%%%%%%%%%%
\section{Connection to the Kolmogorov superposition theorem}\label{sec:hilb}
\input{sec-hilbert}

%%%%%%%%%%%%%%%%%%%%%%%%%%%%%%%%%
\section{Data-driven multivariate model construction}\label{sec:algo}
\input{sec-algo}

%%%%%%%%%%%%%%%%%%%%%%%%%%%%%%%%%
\section{Numerical experiments} \label{sec:exple}
\input{sec-examples}

%%%%%%%%%%%%%%%%%%%%%%%%%%%%%%%%%
\section{Conclusions} \label{sec:conclusion}

We investigated the Loewner framework for linear multivariate/parametric systems and developed a complete methodology (and two algorithms) for data-driven $\ord$-variable pROM realization construction, in the ($\ord$-D) Loewner framework. We also showed the relationship between $\ord$-D Loewner and Sylvester equations. Then, as the numerical complexity and matrix storage explode with the number of data points and variables, we introduce a recursive 1-D null space procedure, equivalent to the full $\ord$-D one. This process allows the decoupling of the variables involved and thus provides the effect of drastically (i) reducing the computational complexity and (ii)  the matrix storage needs. This becomes a major step toward taming the curse of dimensionality. In addition, we have established a connection between the decoupling result and the Kolmogorov superposition theorem (KST). We apply these results to numerical examples throughout the paper, demonstrating their effectiveness. Lastly, we claim that the contributions presented are not limited to the system dynamics and rational approximation fields, but also may apply to many scientific computing areas, including tensor approximation and nonlinear eigenvalue problems, for which dimensionality remains an issue.

\section*{Supplementary material and software availability}
Additional material to supplement the findings reported in this paper is available at
\begin{center}
{\color{blue}\url{https://sites.google.com/site/charlespoussotvassal/nd_loew_tcod}}
\end{center}
and in the report \cite{bigpaper} (in which over 30 test cases are analyzed and various methods are compared in detail).
Furthermore, the \texttt{MATLAB} code used to generate the figures and illustrations corresponding to the numerical results presented in this work is available at

\begin{center}
{\color{blue}\url{https://github.com/cpoussot/mLF}}
\end{center}

%%%%%%%%%%%%%%%%%%%%%%%%%%%%%%%%%

\bibliographystyle{spmpsci}
\bibliography{agp}

%%%%%%%%%%%%%%%%%%%%%%%%%%%%%%%%%
%\newpage\small
%\section*{Appendix}\label{sec:appendix}
%\input{sec-appendix}

\end{document}

%% file: ex_shared.tex
\usepackage{lipsum}
\usepackage{amsfonts}
\usepackage{graphicx}
\usepackage{epstopdf}
\usepackage{algorithmic}
\usepackage{amsopn}
\usepackage{diagbox}
\usepackage{multirow}
\usepackage{makecell}
\usepackage{caption}
\usepackage{subcaption}
\usepackage{longtable}
\usepackage{pdflscape}
\usepackage{adjustbox}
\usepackage{here}
\usepackage{arydshln}
\usepackage{tikz}
\usepackage{tikz-3dplot}
\usetikzlibrary{shapes,arrows}

\ifpdf
  \DeclareGraphicsExtensions{.eps,.pdf,.png,.jpg}
\else
  \DeclareGraphicsExtensions{.eps}
\fi

\title{On the Loewner framework, the Kolmogorov superposition theorem, and the \\ curse of dimensionality}

\author[$\ast$]{A. C. Antoulas}
\affil[$\ast$]{Department of Electrical and Computer Engineering, Rice University, Houston, TX, USA.\authorcr
	\email{aca@rice.edu}}

\author[$\dag$]{I. V. Gosea}
\affil[$\dag$]{Max Planck Institute for Dynamics of Complex Technical Systems, CSC Group \\

	Sandtorstr. 1, 39106 Magdeburg, Germany.\authorcr
	\email{gosea@mpi-magdeburg.mpg.de}, \orcid{0000-0003-3580-4116}}

\author[$\S$]{C. Poussot-Vassal}
\affil[$\S$]{DTIS, ONERA, Universit\'e de Toulouse, 31000, Toulouse, France.\authorcr
	\email{charles.poussot-vassal@onera.fr}, \orcid{0000-0001-9106-1893}}

\abstract{The Loewner framework is an interpolatory approach for the approximation of linear and nonlinear systems. The purpose here is to extend this framework to linear parametric systems with an arbitrary number $\ord$ of parameters. To achieve this, a new generalized multivariate rational function realization is proposed. Then, we introduce the $\ord$-dimensional multivariate Loewner matrices and show that they can be computed by solving a set of coupled Sylvester equations. The null space of these Loewner matrices allows the construction of multivariate rational functions in barycentric form. The principal result of this work is to show how the null space of $\ord$-dimensional Loewner matrices can be computed using a sequence of 1-dimensional Loewner matrices. Thus, a decoupling of the variables is achieved, which leads to a drastic reduction of the computational burden. 
Equally importantly, this burden is alleviated by avoiding the explicit construction of large-scale $\ord$-dimensional Loewner matrices of size $N\times N$. The proposed methodology achieves decoupling of variables, leading to (i) a complexity reduction from ${\cal O}(N^3)$ to below ${\cal O}(N^{1.5})$ when $\ord>5$ and (ii) to memory storage bounded by the largest variable dimension rather than 
their product, thus taming the curse of dimensionality and making the solution scalable to very large data sets. This decoupling of the variables leads to a result similar to the Kolmogorov superposition theorem for rational functions. Thus, making use of barycentric representations, every multivariate rational function can be computed using the composition and superposition of single-variable functions. Finally,  we suggest two algorithms (one direct and one iterative) to construct, directly from data, multivariate (or parametric) realizations ensuring (approximate) interpolation.  Numerical examples highlight the effectiveness and scalability of the method.}

\keywords{parameterized linear systems, Loewner matrix, multivariate functions, barycentric rational interpolation, frequency response, interpolation methods, multivariate barycentric form, Lagrange polynomial basis, multivariate Loewner matrix, Sylvester equations, curse of dimensionality, decoupling of variables, Kolmogorov superposition theorem.}

\novelty{First, we propose a generalized realization for rational functions in $\ord$-variables (for any $\ord$), which are described in the Lagrange basis. Secondly, we show that the corresponding $\ord$-dimensional Loewner matrix can be written as the solution of a series of cascaded Sylvester equations. Then, we demonstrate that the required variables to be determined, i.e., the barycentric coefficients, can be computed using a sequence of $1$-dimensional Loewner matrices instead of the large-scale $\ord$-dimensional one, drastically reducing the computational effort and memory needs. Finally, we show that this decomposition achieves a decoupling of the variables. In this way, a connection between the Loewner framework for rational interpolation of multivariate functions and the Kolmogorov superposition theorem (KST) for the case of rational functions is accomplished. Consequently, this yields a novel formulation of KST, applicable to the special case of rational functions.}

%% file: ex_math.tex
\definecolor{blue}{RGB}{16,97,169} % bleu onera
\definecolor{grey}{RGB}{88, 102, 110} % gris onera 
\definecolor{green}{RGB}{65,209,204} % turquoise
\definecolor{orange}{RGB}{224, 131, 0} % corail
\definecolor{pink}{RGB}{255,105,180} % rose

\def\IR{{\mathbb R}}
\def\IC{{\mathbb C}}

\def\IL{{\mathbb L}}

\def\IV{{\mathbb V}}

\def\IW{{\mathbb W}}
\def\IX{{\mathbb X}}

\def\IN{{\mathbb N}}

\newcommand{\frA}{{\mathcal A}}
\newcommand{\frB}{{\mathcal B}}
\newcommand{\frC}{{\mathcal C}}
\newcommand{\frE}{{\mathcal E}}
\newcommand{\frH}{{\mathcal H}}
\newcommand{\frU}{{\mathcal U}}
\newcommand{\frY}{{\mathcal Y}}
\newcommand{\frX}{{\mathcal X}}

\newcommand{\bC}{{\mathbf C}}

\newcommand{\bG}{{\mathbf B}}
\newcommand{\bS}{{\mathbf S}}

\newcommand{\bL}{{\mathbf L}}
\newcommand{\bM}{{\mathbf M}}

\newcommand{\bI}{{\mathbf I}}
\newcommand{\bH}{{\mathbf H}}

\newcommand{\bW}{{\mathbf W}}
\newcommand{\bR}{{\mathbf R}}
\newcommand{\bT}{{\mathbf T}}

\newcommand{\bX}{{\mathbf X}}
\newcommand{\bx}{{\mathbf x}}
\newcommand{\by}{{\mathbf y}}
\newcommand{\bu}{{\mathbf u}}
\newcommand{\bV}{{\mathbf V}}

\newcommand{\bc}{{\mathbf c}}
\newcommand{\bs}{{\mathbf s}}
\newcommand{\be}{{\mathbf e}}
\newcommand{\bn}{{\mathbf n}}
\newcommand{\bd}{{\mathbf d}}

\newcommand{\bg}{{\mathbf g}}
\newcommand{\bp}{{\mathbf p}}
\newcommand{\br}{{\mathbf r}}
\newcommand{\bq}{{\mathbf q}}

\newcommand{\bt}{{\mathbf t}}
\newcommand{\bv}{{\mathbf v}}
\newcommand{\bw}{{\mathbf w}}

\newcommand{\bpi}{{\mathbf \pi}}

\newcommand{\bfz}{{\mathbf 0}}

\newcommand{\cN}{ {\cal N} }
\newcommand{\cS}{ {\cal S} }

\newcommand{\bPhi}{ \boldsymbol{\Phi} }

\newcommand{\bDelta}{\boldsymbol{\Delta}}
\newcommand{\bLambda}{\boldsymbol{\Lambda}}
\newcommand{\bGamma}{\boldsymbol{\Gamma}}

\newcommand{\bone}{{\mathbf 1}}
\newcommand{\bzer}{{\mathbf 0}}

\newcommand{\rank}{{\mathbf{rank~}}}
\newcommand{\bspan}{{\mathbf{span~}}}
\newcommand{\bvec}{{\mathbf{vec~}}}
\newcommand{\vargen}[2]{\mathstrut^{#2} #1}
\newcommand{\var}[1]{\vargen{s}{#1}}

\newcommand{\IXlag}[1]{\vargen{\mathbb X^{\textrm{Lag}}}{#1}}
\newcommand{\IAlag}{{\mathbb A^{\textrm{Lag}}}}
\newcommand{\IBlag}{{\mathbb B^{\textrm{Lag}}}}

\newcommand{\xlag}[2]{\vargen{\bx_{#2}}{#1}}
\newcommand{\lan}[1]{\vargen{\lambda_{j_{#1}}}{#1}}
\newcommand{\mun}[1]{\vargen{\mu_{i_{#1}}}{#1}}
\newcommand{\lani}[2]{\vargen{\lambda_{{#2}}}{#1}}
\newcommand{\muni}[2]{\vargen{\mu_{{#2}}}{#1}}

\newcommand{\tableau}[1]{\textbf{tab}_{#1}}
\newcommand{\ord}[0]{{n}}
\newcommand{\flop}[0]{\texttt{flop}}

\newcommand{\pare}[1]{\left( #1\right)}
\DeclareMathOperator{\diag}{diag}

\providecommand{\col}[1]{{{#1}}} 
\providecommand{\row}[1]{{{#1}}} 

\newcommand{\bfe}{{\mathbf e}}
\def\IL{\mathbb{L}}

\def\bfz{{\mathbf 0}}

%% file: sec-intro.tex
The context, motivation, and problem statement are first presented. Since it is the principal mathematical tool of the developed method, a brief historical review of Loewner matrix-driven methods is given. Then, the contributions and paper overview are listed.

\subsection{Motivation and context: non-intrusive data-driven model construction}

Rational model approximation addresses the problem of constructing a reduced-order model that captures accurately the behavior of a potentially large-scale model depending on several variables. In the context of dynamical systems governed by differential and algebraic equations, the multivariate nature comes mainly from the parametric dependence of the underlying system or model. These parameters account for physical characteristics such as mass, length, or material properties (in mechanical systems), flow velocity, temperature (in fluid cases), chemical properties (in biological systems), etc. In engineering applications, the parameters are embedded within the model as tuning variables for the output of interest. The challenges and motivation for dynamical multivariate/parametric reduced order model (pROM) construction stem from three inevitable facts about modern computing and engineers' concerns: 
\begin{itemize}
    \item[(i)] First, accurate modeling often leads to large-scale dynamical systems with complex dynamics, for which simulation times and data storage needs become prohibitive, or at least impractical for engineers and practitioners;
    \item[(ii)] Second, the explicit mathematical model describing the underlying phenomena may not be always accessible while input-output data may be measured either from a computer-based (black-box) simulator or directly from a physical experiment; as a consequence, the internal variables of the dynamical phenomena are usually too many to be stored or simply inaccessible; 
    \item[(iii)] Third, a potentially large number of parameters may be necessary to be used in the following steps of the process.
\end{itemize}

Often, complex and accurate parametric models are needed to perform simulations, forecasting, parametric uncertainty propagation, and optimization in a broad sense. The goals are to better understand and analyze the physics, to tune coefficients, to optimize the system, or to construct parameterized control laws. As these objectives often require a multi-query model-based optimization process, the complexity dictates the accuracy, scalability, and applicability of the approach, it is relevant to seek a pROM or multivariate surrogate with low complexity.

\subsection{Literature overview on reduced-order modeling (ROM)}

In the last decade, considerable effort has been dedicated to devising reliable and accurate model reduction (intrusive) and reduced-order modeling (non-intrusive) methods, synthesized in a multitude of approaches developed in the last years \cite{ACA05,morBauBF14,hesthaven2016certified,ABG2020,book2021vol1,book2021vol2}. For the class of parametric systems, the comprehensive review contribution in \cite{morBenGW15} provides an exhaustive account of projection-based methods, from the 2000s until the middle of the 2010s. 

Additionally, relatively new approaches use time-domain snapshot data to compute reduced-order models, such as operator inference (OpInf) \cite{peherstorfer2016data} and dynamic mode decomposition (DMD) \cite{tu2014dynamic}. Extensions of such methods to parameterized dynamical systems have been recently proposed, for OpInf in \cite{yildiz2021learning,mcquarrie2023nonintrusive} and also for DMD in \cite{andreuzzi2023dynamic,sun2023parametric}.  

For the class of frequency-domain methods, we focus on interpolation-based methods. For other classes of projection-based methods, we refer the reader to the survey \cite{morBenGW15}. As explained in this paper, reduced-order models for parametric systems are typically computed employing projection, using either a local or a global basis for matrix or transfer function interpolation. Some contributions in the last years include \cite{morAmsF11,morGeuL16,morYueFB19b,morGosGU21}. %\cite{morPanMEetal10,morAmsF11,morGeuPL13,morGeuL16,morYueFB19b,morGosGU21}.
Additionally,  (quasi-)optimal approaches were proposed in \cite{baur2011interpolatory,morHunMMetal22,mlinaric2024interpolatory}, which aim at imposing optimality in certain norms, e.g., the $\mathcal{H}_2 \otimes \mathcal{L}_2$ norm. 

Non-intrusive methods based on interpolation or approximate matching (using least squares fitting) of transfer function measurements (of the underlying parameterized rational transfer function) have been quite prolific in the last decades, with the following prominent contributions. First (i) extensions of the Loewner Framework (LF) to multivariate rational approximation by interpolation \cite{AIL2012,IA2014,VQPV2023} together with the AAA (Adaptive Antoulas-Anderson) algorithm  for multivariate functions \cite{CBG2023,gosea2022data}. Second (ii) extensions of the Vector Fitting framework to multivariate rational approximation, including the generalized  Sanathanan-Koerner iteration in \cite{bradde2022data,zanco2018enforcing}; these works are mostly concerned with imposing stability and passivity guarantees for macro model generation in the field of electronics.

\subsection{Connection with the Kolmogorov Superposition Theorem}

Problem no. 119, in the book of Polya and Szeg\"o \cite{polya}, asks the question: \textit{Are there actually functions of three variables?} Stated differently: is it possible to use compositions of functions of two or fewer variables to express any function of three variables? This question is related to Hilbert's 13th problem \cite{hilbert2000mathematical}: are there any genuine continuous multivariate functions?  As a matter of fact, Hilbert conjectured the existence of a three-variable continuous function that cannot be expressed in terms of the composition and addition of two-variable continuous functions. For a recent overview of this problem, see \cite{morris}.
The Kolmogorov Superposition Theorem (KST) answers this question negatively.  It shows that continuous functions of several variables can be expressed as the composition and superposition of functions of one variable. Thus, there are no \textit{true} functions of three variables. The present contribution presents connections between the Loewner framework and the KST restricted to rational functions. As a byproduct, \textit{taming the curse of dimensionality}, both in computational complexity, storage, and last but not least, numerical accuracy, is achieved.

\subsection{Connection to other fields}

Tensors are generalizations of vectors and matrices in multiple dimensions. Applications include, among others, the fields of signal processing (e.g., array processing), scientific computing (e.g., multivariate functions discretization), and quantum computing (e.g., simulation of quantum many-body problems). We refer the reader to the survey \cite{kolda2009tensor} for additional information and a detailed discussion. However, explicitly working with tensors, especially of higher dimensions, is not a trivial task. The number of elements in a tensor increases exponentially with the number of dimensions, and so do the computational/memory requirements. The exponential dependency, together with the challenges that arise from it, are connected to the curse of dimensionality (\textbf{C-o-D}).

Tensor decompositions are particularly important and relevant for several strenuous computational tasks since they can potentially alleviate the curse of dimensionality that occurs when working with high-dimensional tensors, as explained in \cite{vervliet2014breaking}. Such a decomposition can accurately represent and substitute the tensor, i.e., one may use it instead of explicitly using the original tensor itself. More details and an extensive literature survey of low-rank tensor approximation techniques, including canonical polyadic decomposition, Tucker decomposition, low multilinear rank approximation, and tensor trains and networks, can be found in \cite{grasedyck2013literature}.

Tensorization and Loewner matrices were recently connected in the contribution \cite{debals2015lowner}. There, a collection of one-dimensional (standard) Loewner matrices is reshaped as a third-dimensional tensor, for which the block term decomposition (BTD) is applied; the procedure is named "Loewnerization". The application of interest is blind signal separation.

Nonlinear eigenvalue problems (NEPs) can be viewed as a generalization of the (ordinary) eigenvalue problem to equations that depend nonlinearly on the parameters. Linearization techniques allow reformulating any polynomial EP as a larger linear eigenvalue problem and then applying the established (classical) algorithms to solve it. Other linearizations involve rational approximation, e.g., \cite{lietaert2022automatic,guttel2022robust}, that involve the usage of the rational Krylov or the AAA algorithms, together with \cite{brennan2023contour}, which uses the Loewner and Hankel frameworks in the context of contour integrals.

\subsection{Problem statement}

A linear-in-state dynamical system parameterized in terms of the parameters of
$\cS=[\var{2},\dots,\var{\ord}]^\top\subset \IC^{\ord-1}$, is characterized in state-space representation by the following equations:
\begin{equation}\label{eq:ss_pLTI}
    %\begin{cases}
        \frE(\cS) \dot{\bx}(t;\cS) = \frA(\cS)\bx(t;\cS) + \frB(\cS)\bu(t),~~
        \by(t;\cS) = \frC(\cS)\bx(t;\cS),
    %\end{cases},
\end{equation}
where $\dot{\bx}(t;\cS)$ refers to the derivative of $\bx(t;\cS) \in \IR^M$, with respect to the time variable $t$. Additionally, the $n_u$ control inputs are collected in the vector $\bu(t) \in \IR^{n_u}$, while the $n_y$ outputs are observed in the vector $\by(t;\cS) \in \IR^{n_y}$. Finally, the dimensions of the system matrices appearing in the state-space realization \cref{eq:ss_pLTI} are as follows: $\frE(\cS), \frA(\cS) \in \IR^{M \times M}$,  $\frB(\cS) \in \IR^{M \times n_u}$, $\frC(\cS) \in \IR^{n_y \times M}$. For simplicity of exposition, we consider only the single-input and single-output (SISO) scenario in what follows, i.e., $n_u=n_y=1$. The extension to multi-input multi-output (MIMO) systems will be left to future research, e.g., based on the formulation in \cite{VQPV2023}. In the sequel, particular attention is allocated to the exposition of a solution that 
\textbf{tames} the curse of dimensionality (\textbf{C-o-D}).

\begin{remark}[Taming the \textbf{C-o-D}]
    Throughout this work, the expression "taming the curse of dimensionality" will be used to emphasize the decoupling of variables, which drastically reduces both (i) the complexity of computation of barycentric weights in terms of \flop, (ii) the memory storage requirements, and (iii) numerical accuracy.
\end{remark}

Transforming the differential equation in \cref{eq:ss_pLTI} using the unilateral Laplace transform, the time variable $t$ becomes $\var{1}$, and solving for the transformed state variable, we have: 
%\begin{equation}\label{eq:XU}
$\frX(\var{1};\cS) = \left[\var{1} \frE(\cS)-\frA(\cS)\right]^{-1} \frB(\cS)\frU(\var{1})$.
%\end{equation}
Similarly, transforming the second equation in \cref{eq:ss_pLTI} we obtain: $\frY(\var{1};\cS) = \frC(\cS)\frX(\var{1};\cS)$. These equations yield
%\begin{equation}\label{eq:YHU}
$\frY(\var{1};\cS) = \frC(\cS) \left[\var{1} \frE(\cS)-\frA(\cS)\right]^{-1} \frB(\cS)\frU(\var{1})$.
%=\bH(\var{1},\cS)\bU(\var{1}),
%\end{equation}
The transfer function of the parametric linear time-invariant (pLTI) system in \cref{eq:ss_pLTI} is given by 
\begin{equation}\label{eq:H}
\frH(\var{1},\var{2},\dots,\var{\ord}) =  \frC(\cS) \left[ \var{1} \frE(\cS) - \frA(\cS)  \right]^{-1} \frB(\cS) \in \IC.
\end{equation}
It is a multivariate rational function involving $\ord$ variables $\var{l}\in \IC$, $l=1,\dots,\ord$.
%including the ones in $\cS$ but also the frequency or Laplace variable, denoted by $\var{1}$. 

We denote the complexity of each variable $\var{l}$ with the value $d_l$ (the highest degree in which the variable occurs in both polynomials describing the rational function shown above) and say that $\frH(\var{1},\var{2},\dots,\var{\ord})$ in \cref{eq:H} is of complexity $(d_1,d_2,\dots,d_\ord)$.

As we are interested in the non-intrusive data-driven setup, let us now consider that the function in \cref{eq:H} is not explicitly known. Instead, one has access to evaluations at (support or interpolatory) points $\lan{1},\lan{2},\dots,\lan{\ord}$ along $\var{1},\var{2},\dots \var{\ord}$, leading to measurements $\bw_{j_1,j_2,\dots,j_\ord}$, for $j_l=1,\ldots,k_l$, where $l=1,\ldots,\ord$. 

Under some assumptions detailed in the sequel, we seek a reduced multivariate rational model, pROM, $\bH$ given as
\begin{equation}\label{eq:Hhat}
 \bH(\var{1},\var{2},\dots,\var{\ord}) = \bC \bPhi(\var{1},\var{2},\dots,\var{\ord})^{-1}\bG \in \IC,
\end{equation}
where the vectors $\bC^\top,\bG\in \IC^{m}$ and square matrix $\bPhi\in \IC^{m \times m}$ define a generalized realization, detailed latter. We denote this realization with the triple $(\bC,\bPhi,\bG)$, being the output, inverse of the resolvent, and input operators.

In the sequel, we concentrate on continuous-time dynamical systems. Therefore, the first variable $\var{1}$ will be associated with the dynamic Laplace one, while $\var{2},\dots,\var{\ord}$ will stand for non-dynamic parametric variables (in most of cases, they will real-valued, although a complex form is also possible). Note that a similar discrete sampled-time model can be obtained using the $z$-transform  (see e.g. \cite{VKP2020}). In addition, one may also notice that \cref{eq:H} may be any multivariate real or complex-valued function.

\subsection{Historical notes}

The Loewner matrix was introduced by Karl L\"owner in the seminal paper published nine decades ago \cite{Loew34}, for the study of matrix convexity. It has been further studied and used in multiple works dealing with data-driven rational function approximation with application in system theory at large. In \cite{AA86}, the Loewner matrix constructed from data is used to compute the barycentric coefficients to obtain the rational approximating function in the Lagrange basis. This is also known as the one-sided Loewner framework. One major update was proposed in 2007 by \cite{MA2007}, introducing the two-sided Loewner framework, constructing a rational model with minimal McMillan degree, and constructing a realization with minimal order, directly from the data. Reference \cite{ALI2017} provides a comprehensive review of the case of single-variable linear systems, gathering most of the results up to 2017. In \cite{AIL2012}, the one-sided framework is extended to two variables/parameters, and its corresponding Lagrange basis realization is derived. Later in \cite{IA2014}, the multi-parameter Loewner framework (mpLF) is presented together (for up to three parameters) with the barycentric form, but without the description of a realization.  Recently, tutorial contributions for the LF, with its extensions and applications, were proposed in \cite{ALI2017,morKarGA19a}. Then, \cite{GPA2022} provides a comprehensive overview including parametric and nonlinear Loewner extensions,  practical applications, and test cases from aerospace engineering and fluid dynamics.

The AAA algorithm in \cite{nakatsukasa2018} represents an iterative and adaptive version of the method in \cite{AA86}, which makes use of the barycentric representation of rational interpolants. For more details on barycentric forms and connections to Lagrange interpolation, we refer the reader to \cite{BT2004barycentric}.
In \cite{CBG2023}, the parametric AAA (p-AAA) algorithm is introduced. This extends the original AAA formulation of \cite{nakatsukasa2018} to multivariate problems appearing in the modeling of parametric dynamical systems. The p-AAA can be viewed as an adaptation of the mpLF, in that it also uses multi-dimensional Loewner matrices and computes barycentric forms of the fitted rational functions. The p-AAA algorithm chooses the interpolation points in a greedy manner and enriches the Lagrange bases until an approximation (with desired accuracy) is reached. 

In addition, multiple application-oriented research papers utilizing the Loewner framework have been suggested, as well as multiple adaptations of the original version. It is worth noticing that the multivariate versions were poorly studied, and if so, limited to three variables. In this paper, we address these two points.

\subsection{Contribution and paper organization}

Our goal is to provide a complete and scalable solution to the data-driven multivariate reduced-order model construction. The results provided in \cite{AIL2012,IA2014} are extended. The main result consists of the decoupling of variables, thus taming the curse of dimensionality. The contribution is five-fold:
\begin{itemize}
    \item[(i)] We propose a multivariate generalized realization allowing to describe with state-space form (with limited complexity), any multivariate rational functions (\Cref{sec:realization} and  \Cref{thm:realization}); 
    \item[(ii)] The $\ord$-D multivariate Loewner matrix is introduced, and is shown to be the solution of a set of coupled Sylvester equations (in \Cref{sec:LL} and \Cref{thm:Loen_sylv});
    \item[(iii)] As the dimension $N$ of the $\ord$-D Loewner matrix exponentially increases with the number of data (i.e. variables and associated degrees), we demonstrate that the associated null space can be obtained using a collection of 1-D Loewner matrices; this leads to the reduction of computational complexity from ${\cal O}(N^3)$ to less than ${\cal O}(N^{1.5})$ when $\ord>5$, to a drastic reduction of storage necessities (\Cref{sec:cod} and, \Cref{thm:cod}, \Cref{thm:complexity} and \Cref{thm:memory}), and to increased numerical accuracy;
    \item[(iv)] A connection with Hilbert's 13th problem and the Kolmogorov Superposition
    Theorem is established  (first with \Cref{thm:decoupling}, and then in  \Cref{sec:hilb});
    \item[(v)] Two data-driven multivariate generalized model construction algorithms in \Cref{sec:algo}, i.e., Algorithm \ref{algo:LL_nD} and Algorithm \ref{algo:LL_nD_adaptive}) are provided;
\end{itemize}

Among these contributions, items (i), (iii), and (iv) are the main theoretical results towards \textbf{taming the curse of dimensionality}, for data-driven multivariate function and realization construction. More specifically, item (i) provides a new realization structure applicable to any $\ord$-dimensional rational function expressed in the Lagrange basis, where the complexity (e.g., dimension of matrices) is controlled. Item (iii) shifts the problem of null space computation of a large-scale $\ord$-D Loewner matrix to the null space computation of a set of small-scale $1$-D Loewner matrices, leading to the very same Lagrange coefficients required in the pROM construction, but with a much lower computational effort. Finally, item (iv) links this result to the Kolmogorov Superposition Theorem by explicitly detailing the decoupling of variables.

\begin{remark}[Connection to tensors]
    Stepping back from the dynamical systems perspective, we also note that the proposed approach provides a candidate solution to tensor approximation problems. Indeed, we approximate any problem characterized by tensorized data sets by means of a rational function. This is done by taming the \textbf{C-o-D} as pointed out in (iii). %thanks to the recursive null space construction in \cref{thm:cod}. 
    Established tensor decompositions may provide a bridge to the philosophy of our proposed method, which requires breaking down the complex problem by eliminating one dimension at every step. 
\end{remark}

\begin{remark}[Connections to NEPs]
The realization proposed addresses the problem of linearization in the context of NEPs. Specifically, %since we are in the multivariate case, 
our realization achieves multi-linearizations of the associated NEPs. Furthermore, in the bivariate case,
if we split the two variables, we achieve a linearization. In the case of more than two variables, if we arrange them as the frequency variable $\var{1}$ in the first group (or right variable), and all the other variables (parameters) in the second group (left variables), we achieve a linearization in $\var{1}$. 
\end{remark}

The remainder of the paper is organized as follows. \Cref{sec:realization} provides the starting point and initial seed by introducing a \textbf{generalized multivariate rational functions realization} framework. From this form,  a specific structure, appropriate to the problem treated here, is chosen. Since \textbf{data/measurements} are the main ingredient of the data-driven framework used, \Cref{sec:data} introduces the data notations, in a general $\ord$-variable case. Then, in \Cref{sec:LL}, the data-based \textbf{$\ord$-D  Loewner matrices} are defined, and a connection with cascaded Sylvester equations is made. The relation with the multivariate barycentric rational form (using a Lagrange basis), as well as the multivariate realization, is also established. In  \Cref{sec:cod}, the numerical complexity induced by the $\ord$-D null space computation is reduced thanks to the decomposition into a recursive set of 1-D Loewner matrix null space computations instead. This decomposition allows a drastic reduction of the complexity, thus \textbf{taming the curse of dimensionality}. 

Finally, \Cref{sec:hilb} details the connection with the Kolmogorov Superposition Theorem. From all these contributions, two algorithms are sketched in \Cref{sec:algo}, indicating complete procedures for the construction of a non-intrusive multivariate dynamical model realization from input-output data.  Numerical examples that illustrate the effectiveness of the proposed process are described in \Cref{sec:exple}\footnote{An exhaustive account of numerical examples and results, together with all necessary data and codes to reproduce the numerics, can be found at the addresses provided at the end of this manuscript.}. Finally, \Cref{sec:conclusion} concludes the paper and provides an outlook on addressing open issues and future research.

%% file: sec-realization.tex
The starting point of this study is the new generalized realization for multivariate rational functions. This leads to the construction of a realization involving internal variable equations, from a $\ord$-variable transfer function in the form \cref{eq:H}. This is expressed in the Lagrange basis. After some preliminaries, the result is stated in \Cref{thm:realization}. This stands as the first major contribution of this paper.

\subsection{Preliminaries}

%Let us first introduce definitions and intermediate results on which we build the generalized realization.

First, we derive the pseudo-companion Lagrange basis, then we provide the multi-row and multi-column indices and coefficient matrices propositions, and finally, results on the characteristic polynomial.

\subsubsection{Pseudo-companion Lagrange matrix}

Consider a rational function $\frH$ in $\ord$ variables, namely $\var{j}$, each of degree $d_j$ ($j=1,\cdots,\ord$), as in \cref{eq:H}. We will consider the \textbf{Lagrange basis} of polynomials. The \textbf{Lagrange pseudo-companion matrix} considered here is denoted $\IXlag{j}$ and is defined as follows.

\begin{definition}\label{def:Xlag}
Let the Lagrange monomials in the variable $\var{j}$ be denoted as $\xlag{j}{i}=\var{j}-\vargen{\lambda_i}{j}$, where $i=1,\cdots,n_j$ and $\vargen{\lambda_i}{j}\in\IC$. Associated with the $j$-th variable, we define the pseudo-companion form  matrix in the Lagrange basis as:
\begin{equation}\label{eq:def:Xlag}%\small
    \IXlag{j} = 
    \left[
    \begin{array}{c}
        \bX^{\textrm{Lag}}(\var{j}) \\ \hline
        \vargen{\bq^{\textrm{Lag}}}{j}
    \end{array}
    \right]=
    \left[
    \begin{array}{ccccc}
        \xlag{j}{1} & -~ \xlag{j}{2} & 0 & \cdots & 0\\
        \xlag{j}{1} & 0 & -~ \xlag{j}{3} & \cdots & 0\\
        \vdots & \vdots & \ddots & \vdots & \vdots \\
        \xlag{j}{1} & 0 & \cdots & 0 &-~\xlag{j}{n_j} \\ \hline 
        \vargen{q_1}{j}&\vargen{q_2}{j} & \cdots & \vargen{q_{n_j-1}}{j} & \vargen{q_{n_j}}{j}
    \end{array}
    \right] \in \IC^{n_j\times n_j}[\var{j}],
\end{equation}
with values $\vargen{q_i}{j}$, $i=1,\ldots,n_j$  chosen so that $\IXlag{j}$ is unimodular, i.e. $\det(\IXlag{j})=1$\footnote{One may chose $1/\vargen{q_i}{j}=\Pi_{k\neq i}(s_i-\lani{j}{k})$ for $k=1,\cdots,n_j$.}.%, thus $\bGamma$ and $\bDelta$ are also unimodular. This is an essential property.
\end{definition}

%{\color{blue}\fbox{
Following the general interpolation framework,
%}}, 
the $\var{j}$ ($j=1,\cdots,\ord$) variables of $\frH$ \cref{eq:H} are split into \textbf{left} and \textbf{right} variables, or equivalently into \textbf{row} and \textbf{column} variables. For simplicity of exposition  (and by permutation, if necessary), we assume that $\var{1},\cdots,\var{k}$ are the column (right) variables and $\var{k+1},\cdots,\var{\ord}$ are the row (left) variables ($0<k<\ord$, $k\in \IN$). Based on this data,  we define two \textbf{Kronecker products} of the associated pseudo-companion matrices:

\begin{definition}\label{def:gamma_delta_lag}
Consider the column $\var{1},\cdots,\var{k}$  and row $\var{k+1},\cdots,\var{\ord}$ variables. We define the Kronecker products of the pseudo-companion matrices \cref{eq:def:Xlag} as 
\begin{equation}\label{eq:def:gamma_delta_lag}
    \begin{array}{rcl}
        \bGamma^{\textrm{Lag}}&=& \IXlag{1}\otimes\IXlag{2} \otimes \cdots \otimes\IXlag{k} \in \IC^{\kappa\times\kappa}[\var{1},\cdots,\var{k}],\\[1mm]
        \bDelta^{\textrm{Lag}} &=& \IXlag{k+1}\otimes \IXlag{k+2} \otimes \cdots \otimes\IXlag{\ord}  \in \IC^{\ell\times\ell}[\var{k+1},\cdots,\var{\ord}],
    \end{array}
\end{equation}
where $\kappa = \prod_{j=1}^{k}n_j$ and $\ell = \prod_{j=k+1}^{\ord}n_j$. These matrices are square and unimodular. For brevity, we will now denote them as $\bGamma$ and $\bDelta$.
\end{definition}

\subsubsection{Multi-row/multi-column indices and the coefficient matrices}

We will show how to set up the matrices containing the coefficients of the numerator and denominator polynomials. The key to this goal is an appropriate definition of row/column multi-indices.

\begin{definition}\label{def:indices}
Each column of $\bGamma$ and each column of  $\bDelta$ defines a unique multi-index $I_q$, $J_r$. We will refer to these indices as \textbf{row-} and \textbf{column-multi-indices}  (the latter, because the $\bDelta$ matrix enters in transposed form), as follows:
\begin{equation*}
I_q=\left[i_{k+1}^q,\,i_{k+2}^q,\,\cdots,\,i_\ord^q\right]
\text{ , }
J_r=\left[j_{1}^r,\,j_{2}^r,\,\cdots,\,j_k^r\right] 
\text{ , $q=1,\cdots,\ell$, $r=1,\cdots,\kappa$}.
\end{equation*}
Each multi-index $I_q$ ($J_r$) contains the indices of the Lagrange monomials involved in the $q$-th ($r$-th) column of $\bDelta$ ($\bGamma$), respectively.
\end{definition}

\begin{remark}
The \textbf{ordering} of these multi-indices is imposed by the ordering of the Kronecker products in \Cref{def:gamma_delta_lag}. More details are available in the examples.
\end{remark}

\subsubsection{The coefficient matrices}

We consider the rational function $\bH$ as
%in the Lagrange basis:
\begin{equation}\label{eq:Hlag}
    \bH(\var{1},\var{2},\cdots,\var{\ord}) =\dfrac{\sum_{j_1=1}^{k_1}\sum_{j_2=1}^{k_2}\cdots \sum_{j_\ord=1}^{k_\ord} \frac{c_{j_1,j_2,\cdots,j_\ord}\bw_{j_1,j_2,\cdots,j_\ord}}{\pare{\var{1}-\lan{1}}\pare{\var{2}-\lan{2}}\cdots\pare{\var{\ord}-\lan{\ord}}}}{\sum_{j_1=1}^{k_1}\sum_{j_2=1}^{k_2} \cdots \sum_{j_\ord=1}^{k_\ord} \frac{c_{j_1,j_2,\cdots, j_\ord}}{\pare{\var{1}-\lan{1}}\pare{\var{2}-\lan{2}}\cdots \pare{\var{\ord}-\lan{\ord}}}},
\end{equation}
where $c_{j_1,j_2,\cdots,j_\ord}\in\IC$ are the barycentric weights and $\bw_{j_1,j_2,\cdots,j_\ord}\in\IC$ the data evaluated at the combination of interpolation (support) points in $\{\lan{1},\lan{2},\cdots,\lan{\ord}\}$, or equivalently, following \Cref{def:indices}, as below
\begin{equation*}
\bH(\var{1},\var{2},\cdots,\var{\ord})=\dfrac{\sum_{q=1}^{\ell}\sum_{r=1}^{\kappa} \frac{\beta_{I_q,J_r}}{\prod_{a\in I_q}\prod_{b\in J_r}(\var{a}-\lan{a})(\var{b}-\lan{b})}}{\sum_{q=1}^{\ell}\sum_{r=1}^{\kappa} \frac{\alpha_{I_q,J_r}}{\prod_{a\in I_q}\prod_{b\in J_r}(\var{a}-\lan{a})(\var{b}-\lan{b})} }
\end{equation*}
We now define matrices of size $\ell\times\kappa$:
\begin{equation}\label{eq:ABlag}\scriptstyle%\small
\IAlag=\left[\begin{array}{cccc}
\alpha_{I_1,J_1}&\alpha_{I_1,J_2}&\cdots&a_{I_1,J_{\kappa}} \\
\alpha_{I_2,J_1}&\alpha_{I_2,J_2}&\cdots&\alpha_{I_2,J_{\kappa}} \\
\vdots&\vdots&\ddots&\vdots\\
\alpha_{I_{\ell},J_1}&\alpha_{I_{\ell},J_2}&\cdots&\alpha_{I_{\ell},J_{\kappa}} \\
\end{array}\right],~
\IBlag=\left[\begin{array}{cccc}
\beta_{I_1,J_1}&\beta_{I_1,J_2}&\cdots&\beta_{I_1,J_{\kappa}} \\
\beta_{I_2,J_1}&\beta_{I_2,J_2}&\cdots&\beta_{I_2,J_{\kappa}} \\
\vdots&\vdots&\ddots&\vdots\\
\beta_{I_{\ell},J_1}&\beta_{I_{\ell},J_2}&\cdots&\beta_{I_{\ell},J_{\kappa}} \\
\end{array}\right].%\in\IC^{\ell\times\kappa}.
\end{equation} 
Notice that $\IAlag$ contains the appropriately arranged barycentric weights of the denominator of $\bH$ (i.e. the entries of a vector in the null space of the associated Loewner matrix), while $\IBlag$, contains the barycentric weights of the numerator, i.e. the product of the denominator barycentric weights with the corresponding values of $\bH$.

\subsubsection{Characteristic polynomial in the barycentric representation}

We consider the single-variable polynomial $\bp(s)$ of degree (at most) $n$ in the variable $s$. 
For a \textbf{barycentric} or \textbf{Lagrange representation}, the following holds (by expanding the determinant of $\bM$ with respect to the last row).

\begin{proposition}\label{prop:polynomial1}
Given the polynomial $\bp(s)$ of degree less than or equal to $n$, 
expressed in a Lagrange basis as
$\bp(s)=\bpi\left(\frac{\alpha_1}{\bx_1}+\cdots+\frac{\alpha_{n+1}}{\bx_{n+1}}\right)$
where $\bpi=\prod_{i=1}^{n+1}\bx_i$.
It follows that $\det(\bM)=\sum_{i=1}^{n+1}\alpha_i\prod_{j\neq i}\bx_j=\bp(s)$, where $\bM$ is the pseudo-companion form matrix as in \Cref{def:Xlag}, where
$\xlag{j}{i}$ is replaced by $\bx_j$ and $\vargen{q_1}{j}$ by $\alpha_j$.
\end{proposition}

Next, following \Cref{prop:polynomial1}, we consider two-variable polynomials $\bp(s,t)$ of degree (at most) $n$, $m$ in the  variables $s$, $t$, respectively. Let  $\bx_i(s)=s-s_i$, $s_i\in\IC$, $i=1,\cdots,n+1$, and $\by_j(t)=t-t_j$,  $t_j\in\IC$, $j=1,\cdots,m+1$, be the monomials constituting a Lagrange basis for two-variable polynomials of degree less than or equal to $n$, $m$, respectively. In other words:\\
%\begin{equation*}
$\bp(s,t)\,=\,\bpi\left[
\frac{\alpha_{1,1}}{\bx_1\by_1}+~\cdots~+\frac{\alpha_{1,m+1}}{\bx_{1}\by_{m+1}}+
\cdots+\frac{\alpha_{n+1,1}}{\bx_{n+1}\by_1}+~\cdots~+\frac{\alpha_{n+1,m+1}}{\bx_{n+1}\by_{m+1}}
\right]$, 
%\end{equation*}
which can be rewritten by highlighting the matrix form of \cref{eq:ABlag}, as,
\begin{equation*}
\bp(s,t)=\bpi\,\left[\frac{1}{\bx_1},\frac{1}{\bx_2},\cdots,\frac{1}{\bx_{n+1}}
\right]\left[
\begin{array}{cccc}
\alpha_{1,1}&\alpha_{1,2}&\cdots&\alpha_{1,m+1}\\
\alpha_{2,1}&\alpha_{2,2}&\cdots&\alpha_{2,m+1}\\
\vdots&\vdots&\ddots&\vdots\\ 
\alpha_{n+1,1}&\alpha_{n+1,2}&\cdots&\alpha_{n+1,m+1}\\
\end{array}\right]\left[\begin{array}{c}\frac{1}{\by_1}\\\frac{1}{\by_1}\\\vdots\\
\frac{1}{\by_{m+1}}\\
\end{array}\right],
\end{equation*}
where $\bpi=\prod_{i=1}^{n+1}\bx_i\prod_{j=1}^{m+1}\by_j$. Consider next, the pseudo-companion form matrices:
\begin{equation}\label{ST}%\scriptstyle%\small
\bS=\underbrace{\left[\begin{array}{ccccc}
\bx_1&-\bx_2& 0 & \cdots & 0\\
\bx_1&0&-\bx_3&\cdots&0\\
\vdots&\vdots&\ddots&\vdots&\vdots\\
\bx_1&0&\cdots&0&-\bx_{n+1}\\ \hline
\epsilon_1&\epsilon_2&\cdots&\epsilon_{n}&\epsilon_{n+1}\\
\end{array}\right]}_{\in\IC^{(n+1)\times(n+1)}[s]},~
\bT=\underbrace{\left[\begin{array}{ccccc}
\by_1&-\by_2& 0 & \cdots & 0\\
\by_1&0&-\by_3&\cdots&0\\
\vdots&\vdots&\ddots&\vdots&\vdots\\
\by_1&0&\cdots&0&-\by_{m+1}\\ \hline
\zeta_1&\zeta_2&\cdots&\zeta_{m}&\zeta_{m+1}\\
\end{array}\right]}_{\in\IC^{(m+1)\times(m+1)}[t]},
\end{equation}
where the constants $\epsilon_i$ and $\zeta_j$ are chosen so that $\det(\bS)=1$ and $\det(\bT)=1$\footnote{One may chose $1/\epsilon_i=\Pi_{j\neq i}(s_i-s_j)$ and  $1/\zeta_i=\Pi_{j\neq i}(t_i-t_j)$, for $i,j=1,\cdots,n,m$.}. The coefficients $\alpha_{i,j}$ are arranged in the form of a matrix $\IAlag{}\in\IC^{(n+1)\times(m+1)}$, as in \cref{eq:ABlag}:
\begin{equation*}
%\begin{array}{rcl}
\IAlag{}=\left[\begin{array}{cccc}
\alpha_{1,1}&\alpha_{1,2}&\cdots&\alpha_{1,m+1}\\
\alpha_{2,1}&\alpha_{2,2}&\cdots&\alpha_{2,m+1}\\
\vdots&\vdots&\ddots&\vdots\\
\alpha_{n+1,1}&\alpha_{n+1,2}&\cdots&\alpha_{n+1,m+1}\\
\end{array}\right],\\
\end{equation*}
or in a vectorized version (taken row-wise) $\mbox{vec}\,(\IAlag{})\in\IC^{1\times\kappa}$ such that 
$$
\mbox{vec}\,(\IAlag{})=\left[\alpha_{1,1},\alpha_{1,2},\cdots\alpha_{1,m+1}\mid\cdots\mid
\alpha_{n+1,1},\cdots,\alpha_{n+1,m+1}\right],
$$
where $\kappa=(n+1)(m+1)$. Consider also the Kronecker product $\bS\otimes\bT$, which turns out to be a square polynomial matrix of size $\kappa$. We form two matrices
\begin{equation}\label{eq:M1M2}
\bM_1=\underbrace{\left[\begin{array}{c}
(\bS\otimes\bT)(1\!:\!\kappa\!-\!1,:) \\[1mm]
\mbox{vec}\,(\IAlag{})\end{array}\right]}_{\in\IC^{\kappa\times\kappa}[s,t]}\text{ and }%,~~
\bM_2=\underbrace{\left[\begin{array}{cc}
\bS(1\!:\!n\!-\!1,:) & ~\bfz_{n-1,m-1}\\
\IAlag{} & ~\bT(1\!:\!m\!-\!1,:)^\top\end{array}\right]}_{\in\IC^{(n+m-1)\times(n+m-1)}[s,t]}.
\end{equation}

\begin{proposition}
The determinants of ~$\bM_1$~ and ~$\bM_2$~ are both equal to ~$\bp(s,t)$.
\end{proposition}

\begin{proof}%\rm
The first expression follows by expanding the determinant of $\bM_1$ with respect to the last row. For the validity of the second expression, see  \Cref{thm:realization}.
\end{proof}

\begin{remark}[Curse of dimensionality]
This result shows that by splitting the variables into left and right, the \textbf{C-o-D} is alleviated, as in the former case the dimension is $(n+1)(m+1)$,  while in the latter the dimension is $n+m-1$.
\end{remark}

\subsection{The multivariate realization in the Lagrange basis}

%The first main result is stated in \Cref{thm:realization}. Its proof is given subsequently.

\subsubsection{Main result}

The result provided in \Cref{thm:realization} yields a systematic way to construct a realization as in \cref{eq:Hhat}, from a transfer function $\bH$ given in a barycentric / Lagrange form \cref{eq:Hlag}.

\begin{theorem}\label{thm:realization} Given quantities in \Cref{def:Xlag} and \Cref{def:gamma_delta_lag}, a $2\ell+\kappa-1=m$-th order realization $(\bC,\bPhi,\bG)$ of the multivariate function $\frH$ in \cref{eq:H}, in barycentric form \cref{eq:Hlag}, satisfying $\bH(\var{1},\cdots,\var{\ord}) = \bC \bPhi(\var{1},\cdots,\var{\ord})^{-1}\bG$, 
is given by,
\begin{align}\label{eq:thm:realization}
%\scriptstyle%\small
\bPhi(\var{1},\cdots,\var{\ord}) &= 
    \left[ \begin{array}{c:c:c}
    \bGamma(1:\kappa-1,:) & \bzer_{\kappa-1,\ell-1} & \bzer_{\kappa-1,\ell} \\ \hdashline
    \IAlag & \bDelta(1:\ell-1,:)^\top & \bzer_{\ell,\ell} \\ \hdashline
    \IBlag &  \bzer_{\ell,\ell-1} & \bDelta^\top 
    \end{array}\right] \in \IC^{m\times m}[\var{1},\cdots,\var{\ord}],     \nonumber
    \\%[5mm]
    \bG &= \left[ \begin{array}{c}
    \bzer_{\kappa-1,1} \\ \hdashline \bDelta(\ell,:)^\top\\ \hdashline\bzer_{\ell,1}
    \end{array}\right] \in\IC^{m\times 1} \mbox{ and }
    \bC = \left[ \begin{array}{c:c:c}
    \bzer_{1,\kappa} & \bzer_{1,\ell-1} & -\be^\top_\ell
\end{array}\right]\in\IC^{1\times m},
\end{align}
where $\be_r$ denotes the $r$-th unit vector (i.e., all entries are zero except the last one, equal to 1) and where $\IAlag, \IBlag\in\IC^{\ell\times \kappa}$ are appropriately chosen, according to the pseudo-companion basis used.
\end{theorem}

\begin{proof}
    See \Cref{sec:proof:realization}.
\end{proof}

\begin{remark}[Matrix realization]
    From \Cref{thm:realization} and following \cref{eq:ss_pLTI}'s notations, it follows that $\bPhi(\var{1},\var{2},\cdots,\var{\ord}) = \var{1}\frE(\cS)-\frA(\cS)$, $\bG = \frB(\cS)$ and $\bC = \frC(\cS)$.
\end{remark}

\begin{corollary}\label{thm:obs_con}
     The realization defined by the tuple $(\bC,\bPhi,\bG)$ has dimension $m = 2\ell+\kappa-1$, and is both R-controllable and R-observable, i.e.,
     \begin{equation}\label{eq:thm:obs_con}
         \left[
         \begin{array}{cc}
            \bPhi(\var{1},\cdots,\var{\ord})  & \bG 
        \end{array}
        \right] 
        \text{ and }
        \left[
        \begin{array}{c}
            \bC \\ \bPhi(\var{1},\cdots,\var{\ord}) 
        \end{array}
        \right] 
    \end{equation}
    have full rank $m$, for all $\var{j}\in\IC$. Furthermore, $\bPhi$ is a polynomial matrix in the variables $\var{j}$ while $\bC$ and $\bG$ are constant.
\end{corollary}

\begin{proof}
The result follows by noticing that the expressions in question have full rank for all $\var{j}\in\IC$, because of the unimodularity of $\bDelta$ and $\bGamma$.
\end{proof}

\subsubsection{Proof of \Cref{thm:realization}} \label{sec:proof:realization}

\paragraph{The numerator of realization \cref{eq:thm:realization}}%\label{proof:num}
First, partition $\bPhi=\left[\begin{array}{c:c}{\bPhi_{11}}&{\bfz}\\\hdashline
{\bPhi_{21}}&{\bPhi_{22}}\end{array}\right]$, where the sizes of the four entries are: $(\kappa+\ell-1)\times(\kappa+\ell-1)$, $(\kappa+\ell-1)\times\ell$, $\ell\times(\kappa+\ell-1)$, $\ell\times\ell$, $\bG=\left[\begin{array}{c}\bG_1\\\hdashline \bfz_{\ell,1}\end{array}\right]$, and  $\bC=[\bfz_{1,\kappa+\ell-1}\underbrace{-\bfe_\ell^\top}_{\bC_2}]$. It follows that
\begin{equation}\label{eq:proof:H}
\bH\,=\,\bC\,\bPhi^{-1}\,\bG=\frac{\bn}{\bd}=\bC_2\,\bPhi_{22}^{-1}\,\bPhi_{21}\,\bPhi_{11}^{-1}\,\bG_1.
\end{equation}
The last expression can be expressed explicitly as:
\small
\begin{equation*}%\scriptstyle%\small
\overbrace{\underbrace{[\bfz_{1,\ell-1} -\bfe_\ell^\top]}_{\bC_2}\underbrace{\bDelta^{-\top}}_{\bPhi_{22}^{-1}}}^{\br_\ell^\top}
\underbrace{[\IBlag \mid \bfz_{\ell,\ell-1}]}_{\bPhi_{21}}~
\overbrace{\underbrace{\left[\begin{array}{c:c}
\bGamma(1:\kappa-1,1:\kappa) & \bfz_{\kappa-1,\ell-1}\\[1mm]\hdashline
\IAlag & \bDelta(1:\ell-1,:)^\top \\
\end{array}\right]^{-1}}_{\bPhi_{11}^{-1}}
\underbrace{\left[\begin{array}{c}
\bfz_{\kappa-1,1}\\[1mm]\hline
\bDelta(\ell,:)^\top\\
\end{array}\right]}_{\bG_1}}^{\left[\begin{array}{c} \bc_\kappa \\\hline\mathbf{\star} \end{array}\right]~
\mbox{where }{\mathbf{\star}}~\mbox{has size}\,\ell-1}.
\end{equation*}
\normalsize
The expressions for $\br_\ell^\top$ and $\bc_\kappa$ are a consequence of \Cref{prop:rc} below. It follows that ${\displaystyle\,\bn=\br_\ell^\top\,\IBlag{}\,\bc_\kappa\,}$.~ Interchanging $\IAlag{}$ and $\IBlag{}$ in \cref{eq:thm:realization}, amounts to  interchanging $\bn$ and $\bd$, in $\bH$ \cref{eq:proof:H}; the expression for the denominator is: ${\displaystyle\,\bd=\br_\ell^\top\,\IAlag{}\,\bc_\kappa\,}$.  
%Despite this fact, we will provide explicit proof in the next paragraph. The above results are based on the following proposition:
%
\begin{proposition}\label{prop:rc} \textbf{(a)} The last row of $\bDelta^{-\top}$ is:
\begin{equation*}%\scriptstyle%\small\footnotesize
\br_\ell^\top=\left[\frac{1}{\vargen{\bx_1}{k+1}},\cdots,~\frac{1}{\vargen{\bx_{n_{k+1}+1}}{k+1}}\right]\otimes
\cdots\otimes \left[\frac{1}{\vargen{\bx_1}{\ord}},\cdots,\frac{1}{\vargen{\bx_{\ord+1}}{\ord}}\right].
\end{equation*}
Therefore $\br^\top_\ell\cdot \IBlag$ is a matrix of size $1\times \kappa$.~ 
\textbf{(b)} The last column of $\bGamma^{-1}$ is:
\begin{equation*}%\scriptscriptstyle:%\small\footnotesize
\bc_\kappa=\left[\frac{1}{\vargen{\bx_1}{1}},\cdots,\frac{1}{\vargen{\bx_{n_{1}+1}}{1}}\right]^\top\otimes
\cdots\otimes \left[\frac{1}{\vargen{\bx_1}{k}},\cdots,\frac{1}{\vargen{\bx_{{n_k}+1}}{k}}\right]^\top.
\end{equation*}
\end{proposition}

\begin{remark}
    The possibility of splitting the variables into \textbf{left} and \textbf{right} variables allows choosing a splitting that \textbf{minimizes} $m$.  For instance, if we have four variables with degrees $(2,2,1,1)$, splitting the variables into $(2,1)$--$(2,1)$ yields $m=17$, while the splitting $(2)$--$(2,1,1)$ (i.e. one column and three rows) yields $m=26$. 
\end{remark}

\subsection{Comments}

In \Cref{thm:realization}, both matrices $\IAlag, \IBlag\in\IC^{\ell\times \kappa}$ are directly related to the pseudo-companion basis chosen in \Cref{def:Xlag} and on the columns-rows variables split. Without entering into technical considerations (out of the scope of this paper), one may notice the following: (i) different pseudo-companion forms \cref{eq:def:Xlag} can be considered, leading to different structures associated with different polynomial bases, such as the Lagrange, or the monomial basis. Here, the Lagrange basis will be considered exclusively; (ii) different permutations and rearrangements of $\var{j}$ in \Cref{def:gamma_delta_lag} may be considered. This results in a different realization order with $m=2\ell+\kappa-1$. Consequently, an adequate choice leads to a reduced order realization, \textbf{taming the realization dimensionality} issue.
 
We are now ready to introduce the main ingredient, namely, the data set. The data can be obtained from any (dynamical) black box model, simulator, or experiment.

%These data are considered to be obtained from any dynamical black box model, simulator, or experiment.

%% file: sec-data.tex
Following the Loewner philosophy presented in a series of papers \cite{MA2007,AIL2012,IA2014,GPA2022}, let us define $P_c$, the column (or right) data, and $P_r$, the row (or left) data. These data will serve the construction of the $\ord$-D Loewner matrices in \Cref{sec:LL}. In what follows, the 1-D 
and 2-D data cases are first recalled, in preparation for the exposition of the general $\ord$-D case.

\subsection{The 1-D case}

When considering single-valued functions $\frH(\var{1})$, i.e., $\ord=1$ in \cref{eq:H}, we define the following column and row data:
\begin{equation} \label{eq:data_1}
P_c^{(1)}:=\left\{\left(\lan{1};\bw_{j_1}\right), ~j_1=1,\ldots,{k_1}\right\},~~
P_r^{(1)}:=\left\{\left(\mun{1};\bv_{i_1}\right), ~i_1=1,\ldots,{q_1}\right\} ,
\end{equation} 
where $\lan{1},\mun{1}\in\IC$ are disjoint interpolation points (or support points), for which the evaluation of $\frH$ respectively leads to $\bw_{j_1}\in\IC$ and $\bv_{i_1}\in\IC$. To support our exposition, let the data \cref{eq:data_1} be represented in the tableau given in  \Cref{tab:1D}, where the measurement vectors $\bW_{k_1}^{\otimes}\in\IC^{k_1}$ and $\bV_{q_1}^{\otimes}\in\IC^{q_1}$ indicate the evaluation of $\frH$ through the single variable $\var{1}$, evaluated at $\lan{1}$ and $\mun{1}$ respectively. \Cref{tab:1D} (also called $\tableau{1}$) is called a measurement  matrix. From $\tableau{1}$, the (1,1) block of dimension $ {k_1}\times 1$ contains the column measurements, and the (1,2) block of dimension ${q_1} \times  1$ contains the row measurements.

\begin{table}[ht!]\small
    \begin{subtable}[b]{.45\linewidth}
    \centering
    \subcaption{$1$-D tableau construction: $\tableau{1}$.}
    \begin{tabular}{|c|c|}\hline
    $\var{1} $&  \\ \hline
    $\lani{1}{1,\cdots,k_1}$ & $\bW_{k_1}^{\otimes}$ \\ \hline 
    $\muni{1}{1,\cdots,q_1}$ & $\bV_{q_1}^{\otimes}$ \\ \hline
    \end{tabular}
    \label{tab:1D}
    \end{subtable}
    \begin{subtable}[b]{.45\linewidth}
    \centering
    \subcaption{$2$-D tableau construction: $\tableau{2}$.}
    \begin{tabular}{|c|c|c|}\hline
    \diagbox{$\var{1}$}{$\var{2}$}  & $\lani{2}{1,\cdots,k_2}$ & $\muni{2}{1,\cdots,q_2}$ \\ \hline
    $\lani{1}{1,\cdots,k_1}$ & $\bW_{k_1,k_2}^{\otimes}$ & $\phi_{cr}$\\ \hline
    $\muni{1}{1,\cdots,q_1}$ & $\phi_{rc}$ & $\bV_{q_1,q_2}^{\otimes}$ \\ \hline
    \end{tabular}
    \label{tab:2D}
    \end{subtable}
    \vspace{-2mm}
\caption{$1$-D and $2$-D tableau construction.}%\label{tab:1-2D}
\end{table}

\subsection{The 2-D case} 

Let us define the column and row data:
\begin{equation} \label{eq:data_2}
\left\{
\begin{array}{rcl}
P_c^{(2)}&:=&\left\{(\lan{1},\lan{2};\bw_{j_1,j_2}), ~j_l=1,\ldots,k_l, \quad l=1,2\right\} \\
P_r^{(2)}&:=&\left\{(\mun{1},\mun{2};\bv_{i_1,i_2}),~i_l=1,\ldots,q_l, \quad l=1,2 \right\} 
\end{array}
\right.,
\end{equation} 
where $\{\lan{1},\mun{1}\}\in\IC^2$ and $\{\lan{2},\mun{2}\}\in\IC^2$ are sets of disjoint interpolation points, for which evaluating $\frH(\var{1},\var{2})$ respectively leads to $\bw_{j_1,j_2},\bv_{i_1,i_2}\in\IC$. Similarly to the 1-D case, data \cref{eq:data_2} is represented in the \Cref{tab:2D}, where $\bW_{k_1,k_2}^{\otimes}\in\IC^{k_1\times k_2}$ and $\bV_{q_1,q_2}^{\otimes}\in\IC^{q_1\times q_2}$ are the measurement matrices related to evaluation of $\frH$ through the two variables $\var{1}$ and $\var{2}$, evaluated at $\{\lan{1},\lan{2}\}$ and $\{\mun{1},\mun{2}\}$.

Compared to the single variable case, the tableau embeds two additional sets of measurements: $\phi_{rc}\in\IC^{q_1\times k_2}$ and  $\phi_{cr}\in\IC^{k_1\times q_2}$. The former resulting from the cross interpolation points evaluation of $\frH(\var{1},\var{2})$ along $\{\mun{1},\lan{2}\}$ and the latter along $\{\lan{1},\mun{2}\}$. It follows that \Cref{tab:2D} (denoted $\tableau{2}$), is a measurement  matrix. 

\begin{remark}[Cross measurements]
    In \cite{IA2014,VQPV2023}, these cross-measurements are used in the extended Loewner matrix construction for improved accuracy. 
\end{remark}

\subsection{The $\ord$-D case}

Now that the single and two variables cases have been reviewed and notations introduced, let us present the $\ord$ variables data case:
\begin{equation} \label{eq:data_n}\scriptscriptstyle%\small
%\left\{\!\!
\begin{array}{l}
P_c^{(\ord)}:=\left\{(\lan{1},\cdots,\lan{\ord};\bw_{j_1,j_2,\cdots,j_\ord}), ~j_l=1,\ldots,k_l, \quad l=1,\cdots,\ord \right\},\\
P_r^{(\ord)}:=\left\{(\mun{1},\ldots,\mun{\ord};\bv_{i_1,i_2,\cdots,i_\ord}),~i_l,\ldots,q_l, \quad l=1,\ldots,\ord \right\}. 
\end{array}
%\right.
\end{equation} 

Similarly, one may derive the $\ord$ variables measurement matrix called $\tableau{\ord}$, illustrated in the table sequence given in \Cref{tab:nD}. Similarly to the expositions made for the single and two-variable cases, each sub-table considers frozen configurations of $\var{3},\var{4},\cdots,\var{\ord}$, along with the combinations of the support points $\lan{3},\lan{4},\cdots,\lan{\ord}$ and $\mun{3},\mun{4},\cdots,\mun{\ord}$, thus forming a $\ord$-dimensional tensor. Particularly, considering the first sub-tableau, the evaluation is for $\var{3},\var{4},\cdots, \var{\ord}=\lani{3}{1},\lani{4}{1},\cdots,\lani{\ord}{1}$. The $\bW_{k_1,k_2,j_3,\cdots,j_\ord}^{\otimes}\in\IC^{k_1\times k_2}$  and $\bV_{q_1,q_2,i_3,\cdots,i_\ord}^{\otimes}\in\IC^{q_1\times q_2}$ entries concatenation form the data tensors $\bW^{\otimes}\in\IC^{k_1\times k_2\times \cdots \times k_\ord}$ and $\bV^{\otimes}\in\IC^{q_1\times q_2\times \cdots \times q_\ord}$; $\tableau{n}$ is an $\ord$-dimensional tensor.

\begin{table}[ht!]\small
    \begin{subtable}[c]{.45\linewidth}
    \centering
    \subcaption{$\col{\var{3}=\lani{3}{1},\var{4}=\lani{4}{1},\dots,\var{\ord}=\lani{\ord}{1}}$.}
    \begin{tabular}{|c|c|c|}\hline
    \diagbox{$\var{1}$}{$\var{2}$}  & $\lani{2}{k_2}$ & $\muni{2}{q_2}$\\ \hline
    $\lani{1}{k_1}$ & $\bW_{k_1,k_2,1,\cdots,1}^{\otimes}$ & $\phi_{cr\col{c \cdots c}}$\\ \hline
    $\muni{1}{q_1}$  & $\phi_{rc\col{c \cdots c}}$ & $\phi_{rr\col{c \cdots c}}$ \\ \hline
    \end{tabular}
    \label{tab:nD1}
    \end{subtable}
    \begin{subtable}[c]{.45\linewidth}
    \centering
    %\begin{flushright}
    \subcaption{$\row{\var{3}=\muni{3}{1},\var{4}=\muni{4}{1}, \dots,\var{\ord}=\muni{\ord}{1}}$.}
    \begin{tabular}{|c|c|c|}\hline
    \diagbox{$\var{1}$}{$\var{2}$}  & $\lani{2}{k_2}$ & $\muni{2}{q_2}$\\ \hline
    $\lani{1}{k_1}$ & $\phi_{cc\row{r\cdots r}}$ & $\phi_{cr\row{r \cdots r}}$\\ \hline
    $\muni{1}{q_1}$ & $\phi_{rc\row{r\cdots r}}$ & $\bV_{q_1,q_2,1,\cdots, 1}^{\otimes}$ \\ \hline
    \end{tabular}
    \label{tab:nD2}
    %\end{flushright}
    \end{subtable}\\
    \begin{subtable}[c]{.45\linewidth}
    \centering
    \subcaption{$\col{\var{3}=\lani{3}{k_3}},\row{\var{4}=\lani{4}{k_4}}, \col{\dots,\var{\ord}=\lani{\ord}{k_\ord}}$.}
    \begin{tabular}{|c|c|c|}\hline
    \diagbox{$\var{1}$}{$\var{2}$}  & $\lani{2}{k_2}$ & $\muni{2}{q_2}$\\ \hline
    $\lani{1}{k_1}$ &  $\bW_{k_1,k_2,\cdots,k_\ord}^{\otimes}$ & $\phi_{cr\col{c \cdots c}}$\\ \hline
    $\muni{1}{q_1}$ & $\phi_{rc\col{c \cdots c}}$ & $\phi_{rr\col{c \cdots c}}$  \\ \hline
    \end{tabular}
    \label{tab:nD3}
    \end{subtable}
    \begin{subtable}[c]{.45\linewidth}
    \centering
    %\begin{flushright}
    \subcaption{$\row{\var{3}=\muni{3}{q_3},\var{4}=\muni{4}{q_4},\dots,\var{\ord}=\muni{\ord}{q_3}}$.}
    \begin{tabular}{|c|c|c|}\hline
    \diagbox{$\var{1}$}{$\var{2}$}  & $\lani{2}{k_2}$ & $\muni{2}{q_2}$\\ \hline
    $\lani{1}{k_1}$ & $\phi_{cc\row{r\cdots r}}$ & $\phi_{cr\row{r\cdots r}}$\\ \hline
    $\muni{1}{q_1}$ & $\phi_{rc\row{r \cdots r}}$ & $\bV_{q_1,q_2,\cdots, q_\ord}^{\otimes}$  \\ \hline
    \end{tabular}
    \label{tab:nDend}
    %\end{flushright}
    \end{subtable}
    \vspace{-2mm}
\caption{$\ord$-D table construction: $\tableau{\ord}$ (some configurations). }
\label{tab:nD}
\end{table}

%% file: sec-LL.tex
Based on \Cref{sec:data} (specifically, on \cref{eq:data_n} and $\tableau{\ord}$), we are now ready to present our \textit{main tool}: the \textbf{multivariate Loewner matrix}. Following the exposition in the previous section, we first recall the 1-D and 2-D Loewner matrices, before presenting the $\ord$-D counterpart. For each dimension, the Loewner matrix is illustrated in close connection to the Sylvester equation that it satisfies. Then, the relation between the Loewner null space and the barycentric rational function is stated, and the connection with generalized realization is established, linking the data of \Cref{sec:data} with the realization of \Cref{sec:realization}.

\subsection{The 1-D case}

The single variable case is briefly mentioned here (more details and connection with dynamical systems theory may be found in \cite{ALI2017}).

\subsubsection{The Loewner matrix and the Sylvester equation}

\begin{definition}\label{def:Loe1}
Given the data described in \cref{eq:data_1}, the 1-D Loewner matrix $\IL_1\in\IC^{q_1\times k_1}$ has $i_1,j_1$-th entries equal to %{\color{blue}\fbox{multi-indices??}}
\begin{equation*}
(\IL_1)_{i_1,j_1} = \dfrac{\bv_{i_1}-\bw_{j_1}}{\mun{1}-\lan{1}},~
i_1=1,\cdots,q_1,~j_1=1,\cdots,k_1.
\end{equation*} 
\end{definition}

\begin{theorem}\label{thm:Loe1_sylv}
Considering the data in \cref{eq:data_1}, we define the following matrices:
\begin{equation*}
\begin{array}{c}
    \bLambda_1 = \diag\pare{\lani{1}{1},\cdots ,\lani{1}{k_1}} \text{ , }
	\bM_1 = \diag\pare{\muni{1}{1},\cdots ,\muni{1}{q_1}} \\[1mm]
    \IW_1 = [\bw_1,\bw_2,\cdots,\bw_{k_1}] \text{ , }
    \IV_1 = [\bv_1,\bv_2,\cdots,\bv_{q_1}]^\top \text{ and }
	\bL_1 = \bone_{q_1},~~\bR_1 = \bone_{k_1}^\top.
\end{array}
\end{equation*}
The Loewner matrix as defined in \Cref{def:Loe1} is the solution of the Sylvester equation: $ \bM_1 \IL_1 - \IL_1 \bLambda_1 = \IV_1 \bR_1 - \bL_1 \IW_1$.
\end{theorem}

\subsubsection{Null space, Lagrange basis form, and generalized realization}

Computing $\IL_1\bc_1=0$, the null space of the Loewner matrix $\IL_1$, the following  holds (with an appropriate number of interpolation points):
$%\begin{equation*}
    \bc_1 = 
    \begin{bmatrix}
        c_1 &  \cdots & c_{k_1}
    \end{bmatrix}^\top \in \IC^{k_1}
$ %\end{equation*}
contains the so-called barycentric weights of the single-variable rational function $\bH(\var{1})$ of degree $(d_1)=(k_1-1)$ given by 
\begin{equation*}
\bH(\var{1}) =
\dfrac{\sum_{j_1=1}^{k_1} \frac{c_{j_1}\bw_{j_1}}{\var{1}-\lan{1}}}{\sum_{j_1=1}^{k_1} \frac{c_{j_1}}{\var{1}-\lan{1}}}
=
\dfrac{\sum_{j_1=1}^{k_1} \frac{\beta_{j_1}}{\var{1}-\lan{1}}}{\sum_{j_1=1}^{k_1} \frac{c_{j_1}}{\var{1}-\lan{1}}},
\end{equation*}
where $\bc_1^\top\odot\IW_1=\left[ \begin{array}{cccc}
\beta_1 & \beta_2 & \cdots &  \beta_{k_1} \end{array}\right]\in\IC^{k_1}$, 
interpolates $\frH(\var{1})$ at points $\lan{1}$. 

\begin{result}[1-D realization]\label{res:real1}
    Given \Cref{def:gamma_delta_lag} and following \Cref{thm:realization}, a generalized realization of $\bH(\var{1})$ is obtained with the following settings: $\IAlag{} = -\bc_1^\top$, $\IBlag{} = \emptyset$, $\bGamma = \IXlag{1}$  and $\bDelta = \emptyset$.
\end{result}

Note that this representation recovers the result already discussed, e.g., in \cite{ALI2017}. 

%A simple example follows below to explicitly illustrate for the reader.

\begin{example} \label{ex1_1D}
    Let us consider $\frH(\var{1})=\frH(s)= (s^2+4)/(s+1)$, being a single-valued rational function of complexity $2$ (i.e. dimension 2 along $s$). By evaluating $\frH$ in $\lan{1}=[1,3,5]$ and $\mun{1}=[2,4,6,8]$, one gets $\bw_{j_1}=[5/2, 13/4, 29/6]$ and $\bv_{i_1}=[8/3, 4, 40/7, 68/9]$. Then, we construct the Loewner matrix, its null space ($\rank\IL_1=2$), and a rational function interpolating the data as,
    \begin{equation*}
        \IL_1 = 
        \left[\begin{array}{ccc} \frac{1}{6} & \frac{7}{12} & \frac{13}{18}\\[1mm] \frac{1}{2} & \frac{3}{4} & \frac{5}{6}\\[1mm] \frac{9}{14} & \frac{23}{28} & \frac{37}{42}\\[1mm] \frac{13}{18} & \frac{31}{36} & \frac{49}{54} \end{array}\right]~~
        \text{ , }
        \bc_1 = 
        \left[\!\begin{array}{r} \frac{1}{3}\\[1mm] -\frac{4}{3}\\[1mm] 1 \end{array}\right]
        \text{ , }
        \bH(s)=
        \frac{\frac{5}{6\,\left(s-1\right)}-\frac{13}{3\,\left(s-3\right)}+\frac{29}{6\,\left(s-5\right)}}{\frac{1}{3\,\left(s-1\right)}-\frac{4}{3\,\left(s-3\right)}+\frac{1}{s-5}}.
    \end{equation*}
    Then, $\bH(s)$ recovers the original function $\frH(s)$. A  realization in the  Lagrange basis can be obtained  as $\bH(s) = \bC\bPhi(s)^{-1}\bG$, where 
    \begin{equation*}
    \bPhi(s) =     
    \left[\begin{array}{ccc} s-1 & 3-s & 0\\[1mm] s-1 & 0 & 5-s\\[1mm] \hdashline 
    -\frac{1}{3} & \frac{4}{3} & -1 \end{array}\right]\text{ and }
    \left\{\!\begin{array}{l}
    \bC = \left[\begin{array}{ccc} \frac{5}{6} & -\frac{13}{3} & \frac{29}{6}
    \end{array}\right] \\[1mm]
    \bG^\top = \left[\begin{array}{cc:c} 0 & 0 & -1 \end{array}\right].
    \end{array} \right.
    \end{equation*}
\end{example}

%%%%%%%%%%
\subsection{The 2-D case}

%We now continue the exposition with the 2-D case.
This section recalls the results originally given in \cite{AIL2012}, for the case of two variables. %The reader is invited to refer to this paper for further details and involved derivations.

\subsubsection{The Loewner matrix and Sylvester equations} 

Similarly to the 1-D case, let us now define the Loewner matrix in the 2-D case.

\begin{definition}\label{def:Loe2}
Given the data described in \cref{eq:data_2}, the $2$-D Loewner matrix $\IL_2\in\IC^{q_1q_2\times k_1k_2}$, has matrix entries given by
\begin{equation*}
\ell_{j_1,j_2}^{i_1,i_2} = \dfrac{\bv_{i_1,i_2}-\bw_{j_1,j_2}}{\pare{\mun{1}-\lan{1}}\pare{\mun{2}-\lan{2}}}.
\end{equation*}
\end{definition}

\begin{definition}\label{def:LoeOtherMat2}
Considering the data given in \cref{eq:data_2}, we define the following  matrices based on Kronecker products:
\begin{equation}\hspace*{-3mm}
\begin{array}{c}
    \bLambda_1 = \diag\pare{\lani{1}{1},\cdots ,\lani{1}{k_1}} \otimes  \bI_{k_2} \text{ , }
	\bM_1 = \diag\pare{\muni{1}{1},\cdots ,\muni{1}{q_1}} \otimes  \bI_{q_2}, \\[1mm]
    \bLambda_2 = \bI_{k_1} \otimes  \diag\pare{\lani{2}{1},\cdots ,\lani{2}{k_2}} \text{ , }
    \bM_2 =  \bI_{q_1} \otimes  \diag\pare{\muni{2}{1},\cdots \muni{2}{q_2}}, \\[1mm]
    \IW_2 = [\bw_{1,1},\bw_{1,2},\cdots,\bw_{1,k_2},\bw_{2,1},\cdots,\bw_{k_1,k_2}] \text{ , }
    \bR_2 = \bone_{k_1k_2}^\top,\\[1mm]
    \IV_2 =  [\bv_{1,1},\bv_{1,2},\cdots,\bv_{1,q_2},\bv_{2,1},\cdots,\bv_{q_1,q_2}]^\top \text{ and }
	\bL_2 = \bone_{q_1q_2}.
\end{array}
\end{equation}
\end{definition}

\begin{theorem}\label{thm:Loe2_sylv}
The 2-D Loewner matrix as defined in \Cref{def:Loe2} is the solution of the  following set of coupled Sylvester equations:
\begin{equation}\label{eq:Sylv2_2eq}
    \bM_2 \IX- \IX \bLambda_2 = \IV_2 \bR_2-\bL_2 \IW_2 \text{ and }
    \bM_1 \IL_{2} - \IL_{2} \bLambda_1  = \IX.
\end{equation}
\end{theorem}

\begin{corollary}
By eliminating the variable $\IX$, it follows that the 2-D Loewner matrix above satisfies the following generalized Sylvester equation: 
\begin{equation*}
\bM_2 \bM_1 \IL_{2} - \bM_2 \IL_{2} \bLambda_1 - \bM_1 \IL_{2} \bLambda_2  + \IL_{2} \bLambda_1 \bLambda_2 = \IV_2 \bR_2-\bL_2 \IW_2.
\end{equation*}
\end{corollary}

\subsubsection{Null space, Lagrange basis form, and generalized realization}

Computing $\IL_2\bc_2=0$, the null space of the bivariate Loewner matrix, we obtain (using the appropriate number of interpolation points):\\[1mm]
$\bc_2^\top = 
\left[
\begin{array}{c|c|c|c}
\underbrace{\begin{array}{c} c_{1,1}~ \cdots ~ c_{1,k_2} \end{array}}_{\alpha_1^\top} &
\underbrace{\begin{array}{c} c_{2,1}~ \cdots ~ c_{2,k_2} \end{array}}_{\alpha_2^\top} & \cdots & \underbrace{\begin{array}{c} c_{k_1,1} ~ \cdots ~ c_{k_1,k_2}\end{array}}_{\alpha_{k_1}^\top} 
\end{array}
\right]\in \IC^{k_1k_2}$,
% = \left[ \begin{array}{c}
%\alpha_1 \\ \hline \alpha_2 \\ \hline \vdots \\ \hline  \alpha_{k_1}
%\end{array}\right] 
and $\bc_2^\top \odot \IW_2 = 
\left[ \begin{array}{c|c|c|c}
\beta_1^\top & \beta_2^\top & \cdots &  \beta_{k_1}^\top \end{array}\right]\in\IC^{k_1k_2}$. These are the barycentric weights of the bivariate rational function $\bH(\var{1},\var{2})$ of degree $(d_1,d_2)=(k_1-1,k_2-1)$:
\begin{equation*}
\bH(\var{1},\var{2}) 
= \dfrac{\sum_{j_1=1}^{k_1}\sum_{j_2=1}^{k_2} \frac{\beta_{j_1,j_2}}{\pare{\var{1}-\lan{1}}\pare{\var{2}-\lan{2}}}}{\sum_{j_1=1}^{k_1}\sum_{j_2=1}^{k_2} \frac{c_{j_1,j_2}}{\pare{\var{1}-\lan{1}}\pare{\var{2}-\lan{2}}}},
\end{equation*}
which interpolates $\frH(\var{1},\var{2})$ at the support points $\{\lan{1},\lan{2}\}$.

\begin{result}[2-D realization]\label{res:real2}
    Given \Cref{def:gamma_delta_lag} and following \Cref{thm:realization}, a generalized realization of $\bg(\var{1},\var{2})$ is obtained by means of $\IAlag = \left[\!\begin{array}{cc} \alpha_1 & \alpha_2~ \cdots ~ \alpha_{k_1} \end{array} \!\right]$, $\IBlag = \left[\!\begin{array}{cc} \beta_1 & \beta_2~ \cdots~ \beta_{k_1}\end{array} \!\right]$, and $\bGamma = \IXlag{1}$, $\bDelta = \IXlag{2}$.
    %\end{equation*}
\end{result}

\begin{example}\label{ex1_2D}
    Let us consider $\frH(\var{1},\var{2})=\frH(s,t) = (s^2t)/(s - t + 1)$  of complexity $(2,1)$. By evaluating $\frH$ in $\lan{1}=[1,3,5]$, $\mun{1}=[0,2,4]$, $\lan{2}=[-1,-3]$ and $\mun{2}=[-2,-4]$ one gets the response tableau $\tableau{2}$ as shown below: 
    \begin{equation*}
        \left[\begin{array}{c|c} \bW_{k_1,k_2}^{\otimes} & \phi_{cr} \\[1mm] \hline \phi_{rc} & \bV_{q_1,q_2}^{\otimes} \end{array}\right] = 
        \left[\begin{array}{cc|cc} -\frac{1}{3} & -\frac{3}{5} & -\frac{1}{2} & -\frac{2}{3}\\[1mm] -\frac{9}{5} & -\frac{27}{7} & -3 & -\frac{9}{2}\\[1mm] -\frac{25}{7} & -\frac{25}{3} & -\frac{25}{4} & -10\\[1mm] \hline 0 & 0 & 0 & 0\\[1mm] -1 & -2 & -\frac{8}{5} & -\frac{16}{7}\\[1mm] -\frac{8}{3} & -6 & -\frac{32}{7} & -\frac{64}{9} \end{array}\right].
    \end{equation*}
%    Then we construct the two-dimensional Loewner matrix and compute its null space, leading to  ($\rank(\IL_2)=5$),
The 2D Loewner matrix is computed with its null space ($\rank(\IL_2)=5$), as
    \begin{equation*}
        \IL_2 = 
        \left[\begin{array}{rr|rr|rr} \frac{1}{3} & -\frac{3}{5} & \frac{3}{5} & -\frac{9}{7} & \frac{5}{7} & -\frac{5}{3}\\[1mm] \frac{1}{9} & \frac{3}{5} & \frac{1}{5} & \frac{9}{7} & \frac{5}{21} & \frac{5}{3}\\[1mm] \frac{19}{15} & -1 & \frac{1}{5} & -\frac{79}{35} & \frac{23}{35} & -\frac{101}{45}\\[1mm] \hline \frac{41}{63} & \frac{59}{35} & -\frac{17}{105} & \frac{11}{7} & \frac{1}{7} & \frac{127}{63}\\[1mm] \frac{89}{63} & -\frac{139}{105} & \frac{97}{35} & -\frac{5}{7} & -1 & -\frac{79}{21}\\[1mm] \frac{61}{81} & \frac{293}{135} & \frac{239}{135} & \frac{205}{63} & -\frac{223}{189} & \frac{11}{9} \end{array}\right]
        \text{ , }
        \bc_2 = 
        \left[\begin{array}{r} -\frac{1}{3}\\[1mm] \frac{5}{9}\\[1mm] \hline \frac{10}{9}\\[1mm] -\frac{14}{9}\\[1mm] \hline  -\frac{7}{9}\\[1mm] 1 \end{array}\right]
        \text{ , }
        \IW_2^\top = 
        \left[\begin{array}{r} -\frac{1}{3}\\[1mm] -\frac{3}{5}\\[1mm] \hline -\frac{9}{5}\\[1mm] -\frac{27}{7}\\[1mm] \hline -\frac{25}{7}\\[1mm] -\frac{25}{3} \end{array}\right].
    \end{equation*}
    It follows that the two-variable rational function $\bH(s,t)$, given in the barycentric form, recovers the original rational function $\frH(s,t)$. Then, a realization in the Lagrange basis (with $\{\lan{1},\lan{2}\}$) is obtained as $\bH(s,t) = \bC\bPhi(s,t)^{-1}\bG$, where $\bC=-\be_{6}^\top$, 
    \begin{equation*}
        \bPhi(s,t) = 
        \left[\begin{array}{ccc:c:cc} s-1 & 3-s & 0 & 0 & 0 & 0\\[1mm] s-1 & 0 & 5-s & 0 & 0 & 0\\[1mm] \hdashline -\frac{1}{3} & \frac{10}{9} & -\frac{7}{9} & t+1 & 0 & 0\\[1mm] \frac{5}{9} & -\frac{14}{9} & 1 & -t-3 & 0 & 0\\[1mm] \hdashline \frac{1}{9} & -2 & \frac{25}{9} & 0 & t+1 & \frac{1}{2}\\[1mm] -\frac{1}{3} & 6 & -\frac{25}{3} & 0 & -t-3 & -\frac{1}{2} \end{array}\right]\text{ and } 
        %\bC^\top=-\be_{6}
        %\left[\begin{array}{c}0\\ 0 \\ 0 \\ 0 \\ 0 \\ \hline -1\end{array}\right] 
        \bG = \left[\begin{array}{c}0 \\ 0 \\ \hdashline 1/2 \\ -1/2 \\ \hdashline 0 \\ 0\end{array}\right].
    \end{equation*}

    \begin{equation*}
        \bC=
        \left[\begin{array}{ccccc|c}0&0&0&0&0&-1\end{array}\right],~
        \bG^\top = 
        \left[\begin{array}{cc|cc|cc}0&0&1/2&-1/2&0&0\end{array}\right].
    \end{equation*}
    Next, by applying the Schur complement, the realization can be 
    compressed to:
    \begin{equation*}
        \bPhi_{\text{c}}(s,t) = 
        \left[\begin{array}{ccc|c} s-1 & 3-s & 0 & 0\\[1mm] s-1 & 0 & 5-s & 0\\[1mm] \hline -\frac{1}{3} & \frac{10}{9} & -\frac{7}{9} & t+1\\[1mm] \frac{5}{9} & -\frac{14}{9} & 1 & -t-3 \end{array}\right] \text{ and }
    \end{equation*}
    \begin{equation*}
        \bC_{\text{c}}(t)=
        \left[\begin{array}{ccc|c} -\frac{2\,t}{9} & 4\,t & -\frac{50\,t}{9} & 0 \end{array}\right],~
        \bG_{\text{c}}^\top=
        \left[\begin{array}{cc|cc} 0&0&\frac{1}{2}&-\frac{1}{2}\end{array}\right].
    \end{equation*}
    The corresponding transfer function perfectly recovers the original one:
    \begin{equation*}
        \hat \bH_{\text{c}}(s,t) =  \bC_{\text{c}}(t) \bPhi_{\text{c}}(s,t)^{-1} \bG_{\text{c}} = \frac{s^2\,t}{s-t+1}.
    \end{equation*}
    Therefore, the realization is further reduced. The price to pay with this compression step is the appearance of a parameter-dependent output matrix $\bC_{\text{c}}(t)$.

%By applying the Schur complement the realization can be compressed to dimension 4, at the expense of introducing a parameter-dependent output matrix $\bC(t)$.%$\bC_{\text{c}}(t)$.
\end{example}

%%%%%%
\subsection{The $\ord$-D case}

This section provides an extension to $\ord$-D Loewner matrices, following the previous two cases, to the scenario involving multivariate functions in $\ord$ variables.

\subsubsection{Loewner matrices and Sylvester equations}

\begin{definition}\label{def:Loen}
Given the data described in \cref{eq:data_n}, the $\ord$-D Loewner matrix $\IL_\ord\in\IC^{q_1q_2\cdots q_\ord\times k_1k_2\cdots k_\ord}$, has entries given by
\begin{equation*}
\ell_{j_1,j_2,\cdots,j_{\ord}}^{i_1,i_2,\cdots,i_{\ord}} = \dfrac{\bv_{i_1,i_2,\cdots,i_\ord}-\bw_{j_1,j_2,\cdots,j_\ord}}{\pare{\mun{1}-\lan{1}}\pare{\mun{2}-\lan{2}}\cdots \pare{\mun{\ord}-\lan{\ord}}}.
\end{equation*}
\end{definition}

\begin{definition}
Considering the data given in \cref{eq:data_n}, we define the following  matrices based on Kronecker products:
\begin{equation*}\hspace*{-3mm}
\begin{array}{ll}
    \bLambda_1 = \diag\pare{\lani{1}{1},\cdots ,\lani{1}{k_1}} \!\otimes\!  \bI_{k_2}\!\otimes\!  \bI_{k_3} \!\otimes\!  \cdots \!\otimes\!  \bI_{k_\ord},&
	\bM_1 = \diag\pare{\muni{1}{1},\cdots ,\muni{1}{q_1}} \!\otimes\!  \bI_{q_2}\!\otimes\!  \bI_{q_3} \!\otimes\!  \cdots \!\otimes\!  \bI_{q_\ord}\\[1mm]
    \bLambda_2 = \bI_{k_1} \!\otimes\!  \diag\pare{\lani{2}{1},\cdots ,\lani{2}{k_2}} \!\otimes\!  \bI_{k_3} \!\otimes\!  \cdots \!\otimes\!  \bI_{k_\ord},&
    \bM_2 =  \bI_{q_1} \!\otimes\!  \diag\pare{\muni{2}{1},\cdots, \muni{2}{q_2}}\!\otimes\!  \bI_{q_3} \!\otimes\!  \cdots \!\otimes\!  \bI_{q_\ord} \\[1mm]
%    \bLambda_3 &=& \bI_{k_1} \!\otimes\!  \bI_{k_2} \!\otimes\!  \diag\pare{\lani{3}{1},\cdots ,\lani{3}{k_3}} \!\otimes\!  \bI_{k_4} \!\otimes\!  \cdots \!\otimes\!  \bI_{k_\ord}\\[1mm]
%    \bM_3 &=&  \bI_{q_1} \!\otimes\! \bI_{q_2} \!\otimes\!  \diag\pare{\muni{3}{1},\cdots, \muni{3}{q_3}} \!\otimes\!  \bI_{q_4} \!\otimes\!  \cdots \!\otimes\!  \bI_{q_\ord}\\[1mm]
    \qquad\qquad\cdots  &\cdots\\[1mm]
    \bLambda_\ord = \bI_{k_1} \!\otimes\!  \cdots \!\otimes\!  \bI_{k_{\ord-1}} \!\otimes\!\diag\pare{\lani{\ord}{1},\cdots ,\lani{\ord}{k_\ord}},&
    \bM_\ord =  \bI_{q_1} \!\otimes\!  \cdots \!\otimes\!   \bI_{q_{\ord-1}}\!\otimes\!\diag\pare{\muni{\ord}{1},\cdots ,\muni{\ord}{q_\ord}},
\end{array}
\end{equation*}
$\IW_\ord = [\bw_{1,1,\cdots,1},\bw_{1,1,\cdots,2},\cdots,\bw_{1,1,\cdots,k_\ord},\bw_{1,\cdots,2,1},\cdots,\bw_{k_1,k_2,\cdots,k_\ord}]$, $\bR_\ord = \bone_{k_1k_2\cdots k_\ord}^\top$,\\ 
$\IV_\ord =  [\bv_{1,1,\cdots,1},\bv_{1,1,\cdots,2},\cdots,\bv_{1,1,\cdots,q_\ord},\bv_{1,\cdots,2,1},\cdots,\bv_{q_1,q_2,\cdots,q_\ord}]^\top$ and 
$\bL_\ord = \bone_{q_1q_2\cdots q_\ord}$.\normalsize
\end{definition}

\begin{theorem}\label{thm:Loen_sylv}
The $\ord$-D Loewner matrix as introduced in \Cref{def:Loen} is the solution of the following set of coupled Sylvester equations:
\begin{equation*}
\left\{\begin{array}{rcl}
    \bM_\ord \IX_1- \IX_1 \bLambda_\ord &=& \IV_\ord \bR_\ord-\bL_\ord \IW_\ord, \\[1mm]
    \bM_{\ord-1} \IX_2- \IX_2 \bLambda_{\ord-1} &=& \IX_1, \\[1mm]
    &\cdots& \\[1mm]
    \bM_{2} \IX_{\ord-1}- \IX_{\ord-1} \bLambda_2 &=& \IX_{\ord-2}, \\[1mm]
    \bM_1 \IL_{\ord} - \IL_{\ord} \bLambda_1  &=& \IX_{\ord-1}.
\end{array}\right.
\end{equation*}
\end{theorem}

\subsubsection{Null space, Lagrange basis form, and generalized realization}

When using an appropriate number of interpolation points, we can compute the null space of the $\ord$-variable Loewner matrix $\IL_{\ord}$, i.e., $\IL_n\bc_n=0$. Here, we denote with $
\bc_\ord^\top =\left[ \begin{array}{cccc}
\alpha_{1} & \mid\mid \alpha_{2} & \mid \cdots&\mid\mid \alpha_{k_1}\end{array}\right] \in\IC^{k_1k_2\cdots k_\ord}$, written in terms of
\begin{equation*}
\hspace*{-2mm}
\begin{array}{l} 
\alpha_1\!=\!\left[\,c_{1,\,\cdots\,,1,1}\,\cdots\, c_{1,\,\cdots\,,1,k_\ord}~
    c_{1,\,\cdots\,,2,1}\,\cdots\, c_{1,\,\cdots\,,2,k_\ord} \mid\cdots \mid 
     c_{1,k_2,\,\cdots\,,k_{\ord-1},1}\,\cdots\, c_{1,k_2,\,\cdots\,,k_{\ord-1},k_\ord}\,\right],\\[1mm]
\alpha_2=\left[\,c_{2,\cdots,1,1}~ \cdots~ c_{2,\cdots,1,k_\ord} ~\mid~ \cdots ~\right],\\
\alpha_{k_1}= \left[\,  c_{k_1,k_2,\cdots,1} ~\cdots ~c_{k_1,k_2,\cdots,k_\ord} \,\right],
\end{array}
\end{equation*}
which are vectors that contain the so-called barycentric weights of the $\ord$-variable rational function $\bH(\var{1},\var{2},\cdots,\var{\ord})$ given by
\begin{equation*}
\bH(\var{1},\var{2},\cdots,\var{\ord}) =\dfrac{\sum_{j_1=1}^{k_1}\sum_{j_2=1}^{k_2}\cdots \sum_{j_\ord=1}^{k_\ord} \frac{\beta_{j_1,j_2,\cdots,j_\ord}}{\pare{\var{1}-\lan{1}}\pare{\var{2}-\lan{2}}\cdots\pare{\var{\ord}-\lan{\ord}}}}{\sum_{j_1=1}^{k_1}\sum_{j_2=1}^{k_2} \cdots \sum_{j_\ord=1}^{k_\ord} \frac{c_{j_1,j_2,\cdots, j_\ord}}{\pare{\var{1}-\lan{1}}\pare{\var{2}-\lan{2}}\cdots \pare{\var{\ord}-\lan{\ord}}}},
\end{equation*}
where $\bc_\ord^\top \odot \IW_\ord = \left[ \begin{array}{cccc} \beta_{1} & \beta_{2} &\cdots & \beta_{k_1} \end{array}\right]\in\IC^{k_1k_2\cdots k_\ord}$. By construction,  the function $\bH(\var{1},\var{2},\cdots,\var{\ord})$ expressed as above interpolates $\frH(\var{1},\var{2},\cdots,\var{\ord})$ at the support points $\{\lan{1},\lan{2},\cdots,\lan{\ord}\}$. 

\begin{result}[$\ord$-D realization, for $k=1$]\label{res:realn}
    Given \Cref{def:gamma_delta_lag} and following \Cref{thm:realization}, a generalized realization of $\bH(\var{1},\var{2},\cdots,\var{\ord})$ is obtained with the following settings: 
    $$
    \IAlag = \left[
    \begin{array}{cccc} 
    \alpha_1 & \alpha_2 & \cdots & \alpha_{k_1}
    \end{array} \right], \IBlag = \left[
    \begin{array}{cccc} 
    \beta_1 & \beta_2 & \cdots & \beta_{k_1}
    \end{array} \right], \bGamma = \IXlag{1}, \bDelta = \IXlag{2}\!\otimes\!  \cdots \otimes\IXlag{\ord}.
    $$
\end{result}

\begin{example}\label{ex1_3D}
Consider the three-variable rational function $\frH(s,t,p) = (s + pt)/(p^2 + s + t)$  of complexity $(1,1,2)$. It is evaluated at $\lan{1}=[2,4]$, $\lan{2}=[1,3]$, $\lan{3}=[5,6,7]$ and $\mun{1}=-\lan{1}$, $\mun{2}=-\lan{2}$, $\mun{3}=-\lan{3}$. The  3D Loewner matrix $\IL_3$ has the property: $\rank(\IL_3)=11$, and the following quantities are put together:
    \begin{equation*}
    \begin{array}{rcl}
        \bc_3^\top &=& 
        \left[\begin{array}{ccc|ccc|ccc|ccc} \frac{1}{2}& -\frac{39}{28}& \frac{13}{14}& -\frac{15}{28}& \frac{41}{28}& -\frac{27}{28}& -\frac{15}{28}& \frac{41}{28}& -\frac{27}{28}& \frac{4}{7}& -\frac{43}{28}& 1 \end{array}\right] \\[2mm]
        \IW_3 &=& 
        \left[\begin{array}{ccc|ccc|ccc|ccc} \frac{1}{4} & \frac{8}{39} & \frac{9}{52} & \frac{17}{30} & \frac{20}{41} & \frac{23}{54} & \frac{3}{10} & \frac{10}{41} & \frac{11}{54} & \frac{19}{32} & \frac{22}{43} & \frac{25}{56} \end{array}\right]
    \end{array}.
    \end{equation*}
    Following \Cref{res:realn}, we may obtain the realization $(\bC,\bPhi(s,t,p),\bG)$. By arranging as $(s)-(t,p)$, one obtains a realization of dimension $m=13$. Instead, by arranging as $(s,t)-(p)$ we get $m=9$. With the latter partitioning, we obtain $\kappa=2\times2$ (associated to variables $s$ and $t$) and $\ell=3$ (associated with variable $p$). Thus, with reference to the multi-indices of \Cref{def:indices}, we readily have $I_1=i_3^1$, $I_2=i_3^2$ and $I_3=i_3^3$, and $J_1=[j_1^1,j_2^1]$, $J_2=[j_1^2,j_2^2]$, $J_3=[j_1^3,j_2^3]$ and $J_4=[j_1^4,j_2^4]$.
    
    By applying the Schur complement, the resolvent can be compressed to size $6$ as $\bPhi_{\text{c}}(s,t,p) =$
    \begin{equation*}
        \left[\begin{array}{cccc|cc} \left(s-2\right)\,\left(t-1\right) & -\left(s-2\right)\,\left(t-3\right) & -\left(t-1\right)\,\left(s-4\right) & \left(s-4\right)\,\left(t-3\right) & 0 & 0\\ 1-\frac{s}{2} & \frac{s}{2}-1 & \frac{s}{2}-2 & 2-\frac{s}{2} & 0 & 0\\ \frac{1}{2}-\frac{t}{2} & \frac{t}{2}-\frac{3}{2} & \frac{t}{2}-\frac{1}{2} & \frac{3}{2}-\frac{t}{2} & 0 & 0\\ \hline \frac{1}{2} & -\frac{15}{28} & -\frac{15}{28} & \frac{4}{7} & p-5 & p-5\\ -\frac{39}{28} & \frac{41}{28} & \frac{41}{28} & -\frac{43}{28} & 6-p & 0\\ \frac{13}{14} & -\frac{27}{28} & -\frac{27}{28} & 1 & 0 & 7-p \end{array}\right]
        %\left[\!\!\begin{array}{cc|ccccc} s-2 & 4-s & 0 & 0 & 0 & 0 & 0\\[1mm] \hline \frac{1}{2} & -\frac{15}{28} & \left(p-5\right)\,\left(t-1\right) & \left(p-5\right)\,\left(t-1\right) & \frac{t}{2}-\frac{1}{2} & \frac{5}{2}-\frac{p}{2} & \frac{5}{2}-\frac{p}{2}\\[1mm] -\frac{39}{28} & \frac{41}{28} & -\left(p-6\right)\,\left(t-1\right) & 0 & 1-t & \frac{p}{2}-3 & 0\\[1mm] \frac{13}{14} & -\frac{27}{28} & 0 & -\left(p-7\right)\,\left(t-1\right) & \frac{t}{2}-\frac{1}{2} & 0 & \frac{p}{2}-\frac{7}{2}\\[1mm] -\frac{15}{28} & \frac{4}{7} & -\left(p-5\right)\,\left(t-3\right) & -\left(p-5\right)\,\left(t-3\right) & \frac{3}{2}-\frac{t}{2} & \frac{p}{2}-\frac{5}{2} & \frac{p}{2}-\frac{5}{2}\\[1mm] \frac{41}{28} & -\frac{43}{28} & \left(p-6\right)\,\left(t-3\right) & 0 & t-3 & 3-\frac{p}{2} & 0\\[1mm] -\frac{27}{28} & 1 & 0 & \left(p-7\right)\,\left(t-3\right) & \frac{3}{2}-\frac{t}{2} & 0 & \frac{7}{2}-\frac{p}{2} \end{array}\!\!\right],
    \end{equation*}
    Additionally, we obtain that $\bG_{\text{c}}^\top=\left[\begin{array}{ccc|ccc}0&0& 0& 1/2&-1&1/2\end{array}\right]$ and the output matrix becomes parameter-dependent $\bC_{\text{c}}(t,p)=\left[\begin{array}{cccc|cc} \frac{p}{28}+\frac{1}{14} & -\frac{3\,p}{28}-\frac{1}{14} & -\frac{p}{28}-\frac{1}{7} & \frac{3\,p}{28}+\frac{1}{7} & 0 & 0 \end{array}\right]$. The transfer function is perfectly recovered:
    \begin{equation*}
        \hat \bH_{\text{c}}(s,t,p) =  \bC_{\text{c}}(t,p) \bPhi_{\text{c}}(s,t,p)^{-1} \bG_{\text{c}} = \frac{s+p\,t}{p^2+s+t}.
    \end{equation*}
\end{example}

\vspace{2mm}

%% file: sec-LLcod.tex
As introduced in \Cref{def:Loen}, it follows that the 
$\ord$-D Loewner matrix $\IL_n$ is of dimension $Q\times K$, where $Q=q_1q_2\dots q_\ord$ and $K=k_1k_2\dots k_\ord$. The dimension increases exponentially with the number of parameters and the corresponding 
degrees (this is also obvious when observing its  Kronecker structure). Therefore,  computing $\bc_\ord$, results in ${\cal O}(QK^2)$ or ${\cal O}(KQ^2)$ \flop \ which stands as a limitation of the proposed approach. It is to be noted that the computationally most favorable case is $K=Q=N$, then the complexity is ${\cal O}(N^3)$ \flop.

The need for the full matrix to perform the \texttt{SVD} decomposition renders the process applicability unfeasible for many data sets. Here, the curse of dimensionality (\textbf{C-o-D}) is addressed through a tailored $\ord$-D Loewner matrix null space decomposition, which results in decoupling the variables. More specifically, in this section, we suggest an alternate approach allowing us to construct $\bc_\ord$ without constructing $\IL_\ord$. This approach \textbf{tames} the \textbf{C-o-D}  by constructing a sequence of 1-D Loewner matrices and computing their associated null space instead. Similarly to the previous sections, for clarity, we start with the 2-D and 3-D cases, before addressing the $\ord$-D case. We finally show that in the $\ord$-D case, the null space boils down to (i) a 1-D Loewner matrix null space and (ii) multiple $(\ord-1)$-D Loewner matrix null spaces. With a recursive procedure, $(n-1)$-D becomes $(n-2)$-D, etc. Thus, this leads to a series of 1-D Loewner matrix null space computations. By avoiding the explicit large-scale $\ord$-D Loewner matrix construction, which is replaced by small-scale 1-D Loewner matrices, it results in drastic \flop \ and storage savings.

\subsection{Null space computation in the 2-D case} %Let us first consider the two variables case. 

\begin{theorem}\label{thm:cod2}
Let $h_{i,j}\in\IC$ be measurements of the transfer function $\frH(\var{1},\var{2})$, 
with $\vargen{s_i}{1}$, $i=1,\cdots,n$, and $\vargen{s_j}{2}$, $j=1,\cdots,m$. Let $k_1= n/2  $ and $k_2= m/2$ be number of column interpolation points (see $P_c^{(2)}$ in \cref{eq:data_2}). The null space of the corresponding 2-D Loewner matrix is spanned by\footnote{We assume here that $n=k_1+q_1$ and $m=k_2+q_2$ and $k_1=q_1$ and $k_2=q_2$. In other configurations, specific treatment is needed.}
\begin{align}
    \cN(\IL_2)=\bvec \left[
    \bc_1^{\lani{2}{1}}\cdot\left[{\bc_1^{\lani{1}{k_1}} }\right]_1,
    \cdots,
    \bc_1^{\lani{2}{k_2}}\cdot\left[\bc_1^{\lani{1}{k_1}}\right]_{k_2} \right],
    \vspace{-6mm}
\end{align}
where $\bc_1^{\lani{1}{k_1}}=\cN(\IL_1^{\lani{1}{k_1}})$ is the null space of the 1-D Loewner matrix for frozen $\var{1}=\lani{1}{k_1}$, and $\bc_1^{\lani{2}{j}}=\cN(\IL_1^{\lani{2}{j}})$ is the $j$-th null space of the 1-D Loewner matrix for frozen $\vargen{s_j}{2}=\{\lani{2}{1},\cdots,\lani{2}{k_2}\}$.
\end{theorem}

\begin{proposition}\label{corr:cod2}
Given the setup in \Cref{thm:cod2}, the null space computation \flop \ complexity is 
$k_1^3+k_1k_2^3$ or $k_2^3+k_2k_1^3$, instead of $k_1^3k_2^3$.
\end{proposition}

\begin{proof}
For simplicity of exposition, let us denote by $h_{i,j}\in\IC$ the value of a transfer function $\frH(s_i,t_j)$. We denote the Lagrange monomials by $s-s_i$, $i=1,\cdots,n$ and by $t-t_j$, $j=1,\cdots ,m$. Then, let the response tableau (the data used for constructing the Loewner matrix) and corresponding barycentric weights $\IAlag$ be defined as 
\begin{equation*}
\underbrace{
\left[\begin{array}{c|cccc}
& t_1&t_2&\cdots&t_m\\ \hline
s_1&   h_{1,1} & h_{1,2} & \cdots & h_{1,m}\\
s_2&   h_{2,1} & h_{2,2} & \cdots &  h_{2,m}\\
\vdots& \vdots & \vdots & \cdots & \vdots\\
s_{n-1}&   h_{n-1,1} & h_{n-1,2} & \cdots &  h_{n-1,m}\\
s_n&   h_{n,1} & h_{n,2} & \cdots & h_{n,m}
\end{array}\right]}_{\tableau{2}} \text{ , }
\underbrace{
\left[\begin{array}{c|cccc}
& t_1&t_2&\cdots&t_m\\\hline
s_1&   \alpha_{1,1} & \alpha_{1,2} & \cdots & \alpha_{1,m}\\
s_2&   \alpha_{2,1} & \alpha_{2,2} & \cdots & \alpha_{2,m}\\
\vdots& \vdots & \vdots & \cdots & \vdots\\
s_{n-1}&   \alpha_{n-1,1} & \alpha_{n-1,2} & \cdots &  \alpha_{n-1,m}\\
s_n&   \alpha_{n,1} & \alpha_{n,2} & \cdots & \alpha_{n,m}
\end{array}\right]}_{\IAlag}.
\end{equation*}

The denominator polynomial in the Lagrange basis is given by $\bd(s,t)=\pi{\sum_{i,j=1}^{n,m}}\frac{\alpha_{i,j}}{(s-s_i)(t-t_j)}$, where $\pi=\Pi_{i=1}^n\Pi_{j=1}^m(s-s_i)(t-t_j)$. The coefficients are given by the null space of the associated 2-D Loewner matrix, i.e. $\cN(\IL_{2})=\bspan(\bvec(\IAlag))$ (where $\IAlag=[\alpha_{i,j}]$).

If now we set $t=t_j$ ($j=1,\cdots,m$), the denominator polynomial becomes $\bd(s,t_j)=\pi_{t_j}{\sum_{i}^{n}}\frac{\alpha_{i,j}}{(s-s_i)}$ where $\pi_{t_j}=\Pi_{i=1}^n(s-s_i)\Pi_{k\neq j}(t_j-t_k)$. In this case, the coefficients are given by the null space of the associated 1-D Loewner matrix, i.e. $\cN(\IL_{1}^{t_j})=\bspan(\left[\alpha_{1,j},\cdots,\alpha_{n-1,j},\alpha_{n,j}\right]^\top)$. Thus, these quantities reproduce the columns of $\IAlag$, up to a constant, for each column.

Similarly, for $s=s_n$, we get $\bd(s_n,t)=\pi_{s_n}{\sum_{j}^{m}}\frac{\alpha_{n,j}}{(t-t_j)}$ where $\pi_{s_n}=\Pi_{i=1}^{n-1}(s_n-s_i)\Pi_{j=1}^m(t-t_j)$. Again, the coefficients are given by the null space of the associated 1-D Loewner matrix, i.e. $\cN(\IL_{1}^{s_n})=\bspan(\left[\alpha_{n,1},\cdots,\alpha_{n,m-1},\alpha_{n,m}\right]^\top)$.

This reproduces the last row of $\IAlag$, up to a constant. To eliminate these constants, we divide the corresponding vectors by $\alpha_{i,m}$ and obtain the following:
\begin{equation*}
\begin{array}{ccccc}
\dfrac{\alpha_{1,1}}{\alpha_{n,1}} & \dfrac{\alpha_{1,2}}{\alpha_{n,2}} & \cdots & 
\dfrac{\alpha_{1,m-1}}{\alpha_{n,m-1}} & \dfrac{\alpha_{1,m}}{\alpha_{n,m}}\\%[1mm]
\dfrac{\alpha_{2,1}}{\alpha_{n,1}} & \dfrac{\alpha_{2,2}}{\alpha_{n,2}} & \cdots & 
\dfrac{\alpha_{2,m-1}}{\alpha_{n,m-1}} & \dfrac{\alpha_{2,m}}{\alpha_{n,m}}\\%[1mm]
\vdots & \vdots & \cdots & \vdots & \vdots\\%[1mm]
\dfrac{\alpha_{n-1,1}}{\alpha_{n,1}} & \dfrac{\alpha_{n-1,2}}{\alpha_{n,2}} & \cdots & 
\dfrac{\alpha_{n-1,m-1}}{\alpha_{n,m-1}} & \dfrac{\alpha_{n-1,m}}{\alpha_{n,m}}\\%[1mm]
1 & 1 & \cdots & 1 & 1\\%[2mm]
\hline\hline \\
%&&&&\\[-3mm]
\dfrac{\alpha_{n,1}}{\alpha_{n,m}}& \dfrac{\alpha_{n,2}}{\alpha_{n,m}}&\cdots&\dfrac{\alpha_{n,m-1}}{\alpha_{n,m}}&1%\\[1mm]
\end{array}
\end{equation*}

Finally, multiplying the $j$-th column with the $j$-th entry of the last row yields a vector that spans the desired null space of $\IL_{2}$. The procedure requires computing the null space of $m$ 1-D Loewner matrices of size $n\times n$ and one 1-D Loewner matrix of size $m\times m$. Consequently, the number of flops is $m n^3+m^3$  instead of $n^3 m^3$, concluding the proof. 
%$\hspace{14mm} \square$

\end{proof}

\begin{remark}[Normalization with other elements]
    In the above treatment, we normalize with the last element of the last row. However, it is clear that normalization with other elements can be chosen. This is especially relevant if the last element is zero, i.e., $\alpha_{n,m}=0$. In such a case, if we choose the $k$-th row, we need the barycentric coefficients of $k$-th first variable.
\end{remark}

\begin{example}\label{ex1_2D_cod}
    Continuing \Cref{ex1_2D}, we construct the tableau with the corresponding values, leading to \Cref{tab:2D_cod_ex}. 
    Here, instead of constructing the 2-D Loewner matrix $\IL_2$ as in \Cref{ex1_2D}, we invoke \Cref{thm:cod2}. We thus construct a sequence of 1-D Loewner matrices as follows\footnote{Here, $n=6$, $m=4$, $k_1=3$, $k_2=2$.}:

    \begin{itemize}
        \item First construct a 1-D Loewner matrix along $\var{1}$, for $\var{2}=\vargen{\lambda_2}{2}=-3$, i.e. considering data of $\tableau{2}(:,2)$ (second column). This leads to 
        \begin{equation*}
            \IL_1^{\vargen{\lambda_2}{2}}=\left[\begin{array}{ccc} -\frac{3}{5} & -\frac{9}{7} & -\frac{5}{3}\\[1mm] -\frac{7}{5} & -\frac{13}{7} & -\frac{19}{9}\\[1mm] -\frac{9}{5} & -\frac{15}{7} & -\frac{7}{3} \end{array}\right]
            \mbox{ and }
            { \bc_1^{\vargen{\lambda_2}{2}}}=    \left[\begin{array}{c} \frac{5}{9}\\[1mm] -\frac{14}{9}\\[1mm] 1 \end{array}\right].
        \end{equation*}
    \item Then, construct three 1-D Loewner matrices along $\var{2}$ for $\var{1}=\{\vargen{\lambda_1}{1},\vargen{\lambda_2}{1},\vargen{\lambda_3}{1}\}$,  i.e. considering data of $\tableau{2}(1,:)$, $\tableau{2}(2,:)$ and $\tableau{2}(3,:)$ (first, second and third rows). This leads to:\\[1mm]
    $\IL_1^{\vargen{\lambda_1}{1}}=\left[\begin{array}{cc} 
    \frac{1}{6} & \frac{1}{10}\\[1mm] \frac{1}{9} & \frac{1}{15} \end{array}\right]$ $\Rightarrow$
    $\bc_1^{\vargen{\lambda_1}{1}}=\left[\!\!\begin{array}{c} -\frac{3}{5}\\[1mm] 1 \end{array}\!\!\right]$,\\[1mm]
    $\IL_1^{\vargen{\lambda_2}{1}}=    \left[\begin{array}{cc} \frac{6}{5} & \frac{6}{7}\\[1mm] \frac{9}{10} & \frac{9}{14} \end{array}\right]$ $\Rightarrow$
    $\bc_1^{\vargen{\lambda_2}{1}}=\left[\!\!\begin{array}{c} -\frac{5}{7}\\[1mm] 1 \end{array}\!\!\right]$,
    $\IL_1^{\vargen{\lambda_3}{1}}=\left[\begin{array}{cc} \frac{75}{28} & \frac{25}{12}\\[1mm] \frac{15}{7} & \frac{5}{3} \end{array}\right]$ $\Rightarrow$
    $\bc_1^{\vargen{\lambda_3}{1}}=\left[\!\!\begin{array}{c} -\frac{7}{9}\\[1mm] 1 \end{array}\!\!\right]$.\\[1mm]
    \item Finally, $\hat \bc_2 = 
    \left[\begin{array}{ccc} 
    \bc_1^{\vargen{\lambda_1}{1}}\cdot [{\bc_1^{\vargen{\lambda_2}{2}}}]_1 &%\\[1mm]
    \bc_1^{\vargen{\lambda_2}{1}}\cdot [{\bc_1^{\vargen{\lambda_2}{2}}}]_2&%\\[1mm]
    \bc_1^{\vargen{\lambda_3}{1}}\cdot [{\bc_1^{\vargen{\lambda_2}{2}}}]_3%\\[1mm]
    \end{array}\right]^\top$, 
    the scaled null space vector is equal to $\bc_2$, directly obtained with the 2-D Loewner matrix (see \Cref{ex1_2D}). Similarly, the rational function and realization follow.\
    \end{itemize}
    The corresponding computational cost is obtained by adding the following \flop: one 1-D Loewner matrix of dimension $3\times 3$ $\rightsquigarrow$ null space computation $3^3=27$ \flop, and three $2\times 2$ 1-D Loewner matrices $\rightsquigarrow$ null space computation is $2^3=8$ \flop. Thus, $27+3\times 8=51$ \flop \ are needed here while $6^3=216$ \flop \ were required in \Cref{ex1_2D}, involving $\IL_2$ directly. Note that the very same result may be obtained by first computing $\cN(\IL_1^{\vargen{\lambda_3}{1}})\in\IR^{2}$, then $\cN(\IL_1^{\vargen{\lambda_1}{2}})$, $\cN(\IL_1^{\vargen{\lambda_2}{2}})\in\IR^{3}$. In this case, the computational cost would be $2^3+2\times 3^3=62$ \flop.
    %In this configuration, the data used are symbolically reported on \cref{tab:2D_cod_ex}.
\end{example}

\begin{table}
\begin{center}
\begin{tabular}{|l|cc|cc|}\hline
%\diagbox{$s_1$}{$s_2$} & $\lani{2}{1}=-1$ & $\lani{2}{2}=-3$ &  $\muni{2}{1}=-2$ & $\muni{2}{2}=-4$ \\[1mm] \hline
\diagbox{$\var{1}$}{$\var{2}$} & $\vargen{\lambda_1}{2} =-1$ & $\vargen{\lambda_2}{2}=-3$ &  $\vargen{\mu_1}{2}=-2$ & $\vargen{\mu_2}{2}=-4$ \\[1mm] \hline
$\vargen{\lambda_1}{1}=1$ & $h_{1,1}=-\frac{1}{3}$ & $h_{1,2}=-\frac{3}{5} $ & $h_{1,3}=-\frac{1}{2}$ & $h_{1,4}=-\frac{2}{3} $\\[1mm] 
$\vargen{\lambda_2}{1}=3$ & $h_{2,1}=-\frac{9}{5}$ & $h_{2,2}=-\frac{27}{7}$ & $h_{2,3}=-3$& $h_{2,4}=-\frac{9}{2}$\\[1mm] 
$\vargen{\lambda_3}{1}=5$ & $h_{3,1}=-\frac{25}{7}$ & $h_{3,2}=-\frac{25}{3}$ & $h_{3,3}=-\frac{25}{4}$& $h_{3,4}=-10$\\[1mm] \hline
$\vargen{\mu_1}{1}=0$ & $h_{4,1}=0$ & $h_{4,2}=0$ & $h_{4,3}=0$& $h_{4,4}=0$\\[1mm] 
$\vargen{\mu_2}{1}=2$ & $h_{5,1}=-1$ & $h_{5,2}=-2$ & $h_{5,3}=-\frac{8}{5} $& $h_{5,4}=-\frac{16}{7}$\\[1mm] 
$\vargen{\mu_3}{1}=4$ & $h_{6,1}=-\frac{8}{3}$ & $h_{6,2}=-6$ & $h_{6,3}=-\frac{32}{7}$& $h_{6,4}=-\frac{64}{9}$\\[1mm] \hline
\end{tabular}
\end{center}
\caption{$2$-D tableau for \Cref{ex1_2D_cod}: $\tableau{2}$.}
\label{tab:2D_cod_ex}
\end{table}

\subsection{Null space computation in the 3-D case} 
%Let us now consider the three variables case. 

\begin{theorem}\label{thm:cod3}
Let $h_{i,j,k}\in\IC$ be measurements of the response of a transfer function $\frH(\var{1},\var{2},\var{3})$, along with $\vargen{s_i}{1}$, $\vargen{s_j}{2}$ and  $\vargen{s_k}{3}$ ($i=1,\cdots,n$, $j=1,\cdots,m$ and $k=1,\cdots,p$), and let $k_1\leq n/2$, $k_2\leq m/2$ and $k_3\leq p/2$ be number of column interpolation points (see $P_c^{(3)}$ in \cref{eq:data_n}, $\ord=3$). The null space of the corresponding 3-D Loewner matrix is spanned by\footnote{We assume here that $n=k_1+q_1$ and $m=k_2+q_2$ and $p=k_3+q_3$, and $k_1=q_1$, $k_2=q_2$ and $k_3=q_3$. In other configurations, specific treatment is needed.}

\begin{align}
\cN(\IL_3)=\bvec \left[
    \bc_2^{\lani{1}{1}}\cdot\left[{\bc_1^{(\lani{2}{k_2},\lani{3}{k_3})}}\right]_1,
    \cdots,
    \bc_2^{\lani{1}{k_1}}\cdot\left[{\bc_1^{(\lani{2}{k_2},\lani{3}{k_3})}}\right]_{k_1}
    \right],
\end{align}
where ${\bc_1^{(\lani{2}{k_2},\lani{3}{k_3})}=\cN(\IL_1^{(\lani{2}{k_2},\lani{3}{k_3})})}$ is the null space of the 1-D Loewner matrix for frozen $\{\var{2},\var{3}\}=\{\lani{2}{k_2},\lani{3}{k_3}\}$, and $\bc_2^{\lani{1}{j}}=\cN(\IL_2^{\lani{1}{j}})$ is the $j$-th null space of the 2-D Loewner matrix for frozen $\vargen{s_j}{1}=\{\lani{1}{1},\cdots,\lani{1}{k_1}\}$.
\end{theorem}

\begin{proof}
    The proof follows that of \Cref{thm:cod2}. First, a 1-D null space Loewner matrix is computed for two frozen variables. Then, 2-D Loewner matrices are computed along with the two other variables. Scaling is similarly applied.
\end{proof}

\begin{remark}[Toward recursivity]
    From \Cref{thm:cod3}, it follows that the 3-D Loewner matrix null space may be obtained from one 1-D Loewner matrix, followed by multiple 2-D Loewner matrices. Then, invoking  \Cref{thm:cod2}, these 2-D Loewner matrix null spaces may be split into a sequence of 1-D Loewner matrix null spaces. Therefore, a recursive scheme naturally appears (see \Cref{ex1_3D_cod}).
\end{remark}

\begin{example}\label{ex1_3D_cod}
We continue with \Cref{ex1_3D}. We now illustrate how much the complexity and dimensionality issue may be reduced when applying the suggested recursive process. First, recall that the three-dimensional Loewner matrix $\IL_3$ is of size $12 \times 12$ and its null space is denoted with $\bc_3^\top = [\bc_{3,1}^\top||\bc_{3,2}^\top]$, given as 
    \begin{equation*}
        \bc_3 =
        \left[\begin{array}{ccc|ccc||ccc|ccc} \frac{1}{2}& -\frac{39}{28}& \frac{13}{14}& -\frac{15}{28}& \frac{41}{28}& -\frac{27}{28}& -\frac{15}{28}& \frac{41}{28}& -\frac{27}{28}& \frac{4}{7}& -\frac{43}{28}& 1 \end{array}\right]^\top.
    \end{equation*}
    Computing such a null space requires an \texttt{SVD} matrix decomposition of complexity $12^3=1,728$ \flop. Here, instead of constructing the 3-D Loewner matrix $\IL_3$ as in \Cref{ex1_3D}, one may construct a sequence of 1-D Loewner matrices, using a recursive approach as follows:
    \begin{itemize}
        \item First, a 1-D Loewner matrix along the first variable $\var{1}$ for frozen second and third variables $\vargen{\lambda_2}{2}=3$ and $\vargen{\lambda_3}{3}=7$, i.e. elements of $\tableau{3}(:,2,3)$, leading to
        \begin{equation*}
        \IL_1^{(\vargen{\lambda_2}{2},\vargen{\lambda_3}{3})}=\left[\begin{array}{cc} \frac{31}{2700} & \frac{31}{2800}\\[1mm] \frac{31}{2592} & \frac{31}{2688} \end{array}\right]
        \text{ and }
        {\bc_1^{(\vargen{\lambda_2}{2},\vargen{\lambda_3}{3})}}=\left[\begin{array}{c} -\frac{27}{28}\\[1mm] 1 \end{array}\right].
        \end{equation*}
        \item Second, as $\vargen{\lambda_j}{1}$ is of dimension two ($k_1=2$), two 2-D Loewner matrices appear: one for frozen $\vargen{\lambda_1}{1}$ and one for frozen $\vargen{\lambda_2}{1}$, along $\var{2}$ and $\var{3}$, i.e. elements of $\tableau{3}(1,:,:)$ and $\tableau{3}(2,:,:)$. The first and second 2-D Loewner matrices lead to null spaces spanned by: \small
        \begin{equation*}
            \bc_2^{\vargen{\lambda_1}{1}}\!=\!
            \left[\!\frac{-14}{27},~ \frac{13}{9}, ~ \frac{-26}{27}, ~ \frac{5}{9},
            ~\frac{-41}{27},~ 1\!\right]^\top\!,
            \bc_2^{\vargen{\lambda_2}{1}}\!=\!
            \left[\!\frac{-15}{28},~\frac{41}{28},~\frac{-27}{28},~\frac{4}{7},~\frac{-43}{28},~ 1 \!\right]^\top,
        \end{equation*}
        \normalsize
        which can now be scaled by the coefficients of $\bc_1^{(\vargen{\lambda_2}{2},\vargen{\lambda_3}{3})}$, leading to
        \begin{equation*}
        \hat \bc_3=\left[\begin{array}{cc} 
        \bc_2^{\vargen{\lambda_1}{1}} \cdot  [{ \bc_1^{(\vargen{\lambda_2}{2},\vargen{\lambda_3}{3})}}]_{1} & %\\[1mm] 
        \bc_2^{\vargen{\lambda_2}{1}} \cdot [{ \bc_1^{(\vargen{\lambda_2}{2},\vargen{\lambda_3}{3})}}]_{2}
        \end{array}\right]^\top
        =\bc_3 .
        \end{equation*}
    \end{itemize}
    By considering the first 2-D Loewner matrix $\IL_2^{\vargen{\lambda_1}{1}}$ leading, to the null space $\bc_2^{\vargen{\lambda_1}{1}}$, the very same process as the one presented in the previous subsection (2-D case) may be performed (to avoid the 2-D matrix construction). In what follows we describe this iteration (for $\bc_2^{\vargen{\lambda_1}{1}}$ only, as it similarly apply to $\bc_2^{\vargen{\lambda_2}{1}}$).
    \begin{itemize}
        \item First, one constructs the 1-D Loewner matrix along the second variable $\var{2}$ for frozen first and third variables, i.e., elements of $\tableau{3}(1,:,3)$, leading to
        \begin{equation*}
            \IL_1^{(\vargen{\lambda_1}{1},\vargen{\lambda_3}{3})}= \left[\begin{array}{cc} \frac{71}{520} & \frac{71}{540}\\[1mm] \frac{355}{2496} & \frac{355}{2592} \end{array}\right]
            ~~\textrm{and}~~
            \bc_1^{(\vargen{\lambda_1}{1},\vargen{\lambda_3}{3})}= \left[\begin{array}{c} -\frac{26}{27}\\[1mm] 1 \end{array}\right].
        \end{equation*}
        \item Second, as $\vargen{\lambda_{k_2}}{2}$ is of dimension two ($k_2=2$), two 1-D Loewner matrices appear: one for frozen $\vargen{\lambda_1}{2}$ and one for frozen $\vargen{\lambda_2}{2}$, along $\var{3}$ (here again, $ \var{1}$ is frozen to $\vargen{\lambda_1}{1}$). The first and second 1-D Loewner matrices lead to the following null spaces,
\begin{equation*}
   \bc_1^{(\vargen{\lambda_1}{1},\vargen{\lambda_1}{2})}= 
   \left[\begin{array}{ccc} \frac{7}{13}& -\frac{3}{2}& 1 \end{array}\right]^\top~~
\mbox{and}~~
\bc_1^{(\vargen{\lambda_1}{1},\vargen{\lambda_2}{2})}= \left[\begin{array}{ccc} 
\frac{5}{9}& -\frac{41}{27}& 1 \end{array}\right]^\top,
        \end{equation*}
        which can now be scaled by the coefficients of $\bc_1^{(\vargen{\lambda_1}{1},\vargen{\lambda_3}{3})}$, leading to
        \begin{equation*}
        \left[\begin{array}{c} 
        \bc_1^{(\vargen{\lambda_1}{1},\vargen{\lambda_1}{2})} \cdot [\bc_1^{(\vargen{\lambda_1}{1},\vargen{\lambda_3}{3})}]_{1}\\[1mm]
        \bc_1^{(\vargen{\lambda_1}{1},\vargen{\lambda_2}{2})}\cdot [\bc_1^{(\vargen{\lambda_1}{1},\vargen{\lambda_3}{3})}]_{2}
        \end{array}\right] 
        =\bc_2^{\vargen{\lambda_1}{1}}.
        \end{equation*}
        By scaling $\bc_2^{\vargen{\lambda_1}{1}}$ with the first element of $\bc_1^{(\vargen{\lambda_2}{2},\vargen{\lambda_3}{3})}$ then leads to $\bc_{3,1}^\top$.
    \end{itemize}
    
    This step is repeated for $\IL_2^{\vargen{\lambda_2}{1}}$ leading, to the null space $\bc_2^{\vargen{\lambda_2}{1}}$. The later is scaled with the second element of $\bc_1^{(\vargen{\lambda_2}{2},\vargen{\lambda_3}{3})}$,  leading to $\bc_{3,2}^\top$. By checking the complexity, one observes that only a collection of 1-D Loewner matrices needs to be constructed, as well as their null spaces. Here,  (i) one 1-D Loewner matrix along $\var{1}$ of dimension $2\times 2$ and (ii) two 2-D Loewner matrices along $\var{2}$ and $\var{3}$, recast as, two 1-D Loewner matrices along $\var{2}$ of dimension $2\times 2$, and four 1-D Loewner matrices along $\var{3}$ of dimension $3\times 3$. The resulting complexity is $(1\times 2^3)+(2\times 2^3)+(4\times 3^3)=132$ \flop, 
    being much less than 1,728 \flop \ for $\IL_3$. One may also notice that changing the variables orders as $\var{1}\leftarrow \var{3}$ and $\var{3}\leftarrow \var{1}$ would lead to $(1\times 3^3)+(3\times 2^3)+(6 \times 2^3)=99$. In both cases, the multi-variate Loewner matrices are no longer needed and can be replaced by a series of single variables, taming the curse of dimensionality.
\end{example}

\subsection{Null space computation in the $\ord$-D case and variable decoupling}

We now state the second main result of this paper: \Cref{thm:cod} and \Cref{thm:decoupling}, allowing us to address the \textbf{C-o-D} related to the null space computation of the $\ord$-D Loewner matrix. This is achieved by splitting a $\ord$-D Loewner matrix null space into a 1-D and a collection of $(\ord-1)$-D null spaces, thus another sequence of 1-D and $(\ord-2)$-D null spaces, and so on \dots 

\begin{theorem}\label{thm:cod}
Given the tableau $\tableau{\ord}$ as in \Cref{tab:nD} being the evaluation of the $\ord$-variable $\frH$ function \cref{eq:H} at the data set \cref{eq:data_n}, the null space of the corresponding $\ord$-D Loewner matrix $\IL_\ord$, is spanned by
\begin{equation*}
    %\cN(\IL_\ord)=
    \bvec \left[
    \bc_{\ord-1}^{\lani{1}{1}}\cdot\left[{ \bc_1^{(\lani{2}{k_2},\lani{3}{k_3},\cdots,\lani{\ord}{k_\ord})}}\right]_1,
    \cdots,
    \bc_{\ord-1}^{\lani{1}{k_1}}\cdot\left[{ \bc_1^{(\lani{2}{k_2},\lani{3}{k_3},\cdots,\lani{\ord}{k_\ord})}}\right]_{k_1}
    \right],
\end{equation*}
where $ \bc_1^{(\lani{2}{k_2},\lani{3}{k_3},\cdots,\lani{\ord}{k_\ord})}$ spans
$\cN(\IL_1^{(\lani{2}{k_2},\lani{3}{k_3},\cdots,\lani{\ord}{k_\ord})})$, i.e. the nullspace 
of the 1-D Loewner matrix for frozen $\{\var{2},\var{3},\cdots,\var{\ord}\}=\{\lani{2}{k_2},\lani{3}{k_3},\cdots,\lani{\ord}{k_\ord}\}$, and $\bc_{\ord-1}^{\lani{1}{j}}$ spans 
$\cN(\IL_{\ord-1}^{\lani{1}{j}})$, i.e. the $j$-th null space of the $(\ord-1)$-D Loewner matrix for frozen $\vargen{s_j}{1}=\{\lani{1}{1},\cdots,\lani{1}{k_1}\}$.
\end{theorem}

\begin{proof}
    The proof follows the one given for the 2-D and 3-D cases.
\end{proof}

\Cref{thm:cod} provides a means to compute the null space of an $\ord$-D Loewner matrix via a 1-D and $k_1$, $(\ord-1)$-D Loewner matrices. Evidently, the latter $(\ord-1)$-D Loewner matrix null spaces may also be obtained by $k_1$, 1-D Loewner matrices plus $k_1k_2$, $(\ord-2)$-D Loewner matrices. This reveals a recursive scheme that splits the $\ord$-D Loewner matrix into a set of 1-D Loewner matrices. As a consequence, the following decoupling theorem holds.
\begin{theorem} \label{thm:decoupling}
Given data \cref{eq:data_n} and \Cref{thm:cod}, the latter achieves decoupling of the variables, and the null space can be equivalently written as:
    \begin{equation} \label{eq:decoupling}
        \bc_\ord = \underbrace{\bc^{\var{\ord}}}_{\mathbf{Bary}({\var{\ord})}} \odot 
         \underbrace{(\bc^{\var{\ord-1}} \otimes {\bone}_{k_\ord} )}_{\mathbf{Bary}({\var{\ord-1}})} \odot
        \underbrace{(\bc^{\var{\ord-2}} \otimes {\bone}_{k_\ord k_{\ord-1}} )}_{\mathbf{Bary}({\var{\ord-2}})} \odot 
        \cdots \odot 
        \underbrace{(\bc^{\var{1}} \otimes {\bone}_{k_{\ord}\dots k_2} )}_{\mathbf{Bary}({\var{1}})}.
    \end{equation}
where $\bc^{\var{l}}$ denotes the vector of barycentric coefficients related to the $l$-th variable.
\end{theorem}
As an illustration, in \Cref{thm:decoupling}, $\bc^{\var{1}}=\bc_1^{(\lani{2}{k_2},\lani{3}{k_3},\cdots,\lani{\ord}{k_\ord})}$ while $\bc^{\var{2}}$ is the vectorized collection of $k_1$ vectors $\bc_1^{(\lani{1}{1},\lani{3}{k_3},\cdots,\lani{\ord}{k_\ord})}, \cdots, \bc_1^{(\lani{1}{k_1},\lani{3}{k_3},\cdots,\lani{\ord}{k_\ord})} $ and so on. In \Cref{sec:hilb} and \cref{eq:kst_decoupling}, an illustrative numerical example is given. Next, we assess how much this contributes to taming the \textbf{C-o-D}, both in terms of \flop \ and memory savings.

\subsection{Summary of complexity, memory requirements, and accuracy}

Let us now state the main complexity result, related to \Cref{thm:cod}. The latter is stated in \Cref{thm:complexity} and  \Cref{thm:memory} being the two major justifications for taming the \textbf{C-o-D}. They state the drastic reduction of the computational complexity and the required memory.

\begin{theorem}\label{thm:complexity}
    The  \flop \ count for the recursive approach \Cref{thm:cod}, is:
    \begin{equation}
    %\left\{
    %\begin{array}{rcl}
        \text{\flop}_{1} = \displaystyle\sum_{j=1}^\ord \left( k_j^3 \prod_{l=1}^{j} k_{l-1}\right)\text{ where } k_0= 1,
    \end{equation}
\end{theorem}

\begin{proof}
    Consider a function in $\ord$ variables $\var{l}$, of degree $d_l>1$, $l=1,\cdots,\ord$ (and let $k_l=d_l+1$). \Cref{tab:complexity} shows the complexity as a function of the no. of variables.
    \begin{table}[ht!]
        \begin{center}
        \begin{tabular}{|c|c|c|c|}\hline
        $\#$ of variables of $\frH$ & $\#\IL_{1}$ matrix & Size of each $\IL_{1}$  & \flop \ per $\IL_{1}$ \\
        \hline
        $\ord$ & $k_1\,k_2\,\cdots \,k_{\ord-2}\,k_{\ord-1}$ & $k_\ord$ & $k_\ord^3$\\[1mm]
        $\ord-1$ &$k_1\,k_2\,\cdots \,k_{\ord-2}$ & $k_{\ord-1}$ & $k_{\ord-1}^3$\\[.5mm]
        \vdots&\vdots&\vdots&\vdots\\[.5mm]
        3 & $k_1\,k_2$ & $k_{3}$ & $k_{3}^3$\\[1mm]
        2 & $k_1$ & $k_{2}$ & $k_{2}^3$\\[1mm]
        1 & $1$ & $k_{1}$ & $k_{1}^3$\\ \hline
        \end{tabular}
        \end{center}
            \vspace{-2mm}
        \caption{Depiction of complexity as a function of the number of variables.}
        \label{tab:complexity}
    \end{table}

    Hence, the total number of \flop \ required to compute an element of the null space of the $\ord$-D Loewner matrix $\IL_{\ord}$ is:
    \begin{equation*}
        \begin{array}{rcl}
        \flop_{1}
        &=&k_1^3+\left(k_1\right)k_2^3+\cdots+\left(k_1k_2\cdots k_{\ord-2}\right)\,k_{\ord-1}^3+ \left(k_1k_2\cdots k_{\ord-2}k_{\ord-1}\right)\,k_\ord^3\\[1mm]
        &=&k_1^3+k_1\,\left(k_2^3+k_2\,\left(k_3^3+\cdots k_{\ord-2} \left(k_{\ord-1}^3+k_{\ord-1\,}\left(k_\ord^3\right)\,\right)\,\right)\,\right).
    \end{array}    
    \end{equation*}
    %This concludes the proof.
\end{proof}

\begin{corr}
The variable arrangement that minimizes the \flop \ cost is the one obtained by re-ordering each variable $\var{l}$ in decreasing complexity order $d_l$, i.e. $d_{l} \geq d_{l+1}$, for $l=1,\cdots,\ord-1$.
\end{corr}

\begin{corr}
    The most computationally demanding configuration occurs when each $\var{l}$ order satisfies $d_l=k_l-1=k-1$ ($l=1,\ldots,\ord$), requiring $k$ interpolation points each. The worst case \flop \ writes (note that $N=k^\ord$)
    \begin{equation}\label{eq:flop_worst}
    \overline{\text{\flop}_{1}}=k^3+k^4+\cdots+k^{n+2} 
    = k^3\dfrac{1-k^{n}}{1-k}=k^3\dfrac{1-N}{1-k},
    \end{equation}
\end{corr}
Note that \cref{eq:flop_worst} is a ($n$ finite) geometric series of ratio $k$. Consequently, an upper bound of \cref{eq:flop_worst} can be estimated by considering that $k>1$ and for a different number of variables $\ord$. As an example, for $\ord=\{1,2,3,4,\cdots\}$, the complexity is upper bounded by $\{{\cal O}(N^{3}),{\cal O}(N^{2.30}),{\cal O}(N^{1.94}),{\cal O}(N^{1.73}),\cdots\}$ respectively. One can clearly observe that when the number of variables $\ord>1$, the flop complexity drops to $2.30$, and this decreases as $\ord$ increases. For example, when $\ord=50$, one obtains ${\cal O}(N^{1.06})$.

In \Cref{fig:complexity}, we show the result in \Cref{thm:complexity} (cascaded $\ord$-D Loewner) and compare it to the reference full $\IL_\ord$ null space computation via \texttt{SVD}\footnote{One should note that we are considering here the case $K=Q=N$ to simplify the exposition.}, of complexity ${\cal O}(N^3)$ and with ${\cal O}(N^2)$ and ${\cal O}(N\log(N))$ references. In the same figure, we evaluate the worst case \cref{eq:flop_worst} for different numbers of considered variables $\ord=\{1,2,\cdots,50\}$ (each is evaluated with complexity $k=1,\cdots,50$). Then, we evaluate an upper complexity approximate of the form ${\cal O}(N^x)$, where $x>0$, to be an upper bound of the data set.

With similar importance, the data storage is a key element in the \textbf{C-o-D}. In complex and double precision, the construction of the $\ord$-D Loewner matrix $\IL_\ord \in \IC^{N\times N}$, where $N=k_1k_2\cdots k_n$, requires a disk storage of $\frac{8}{2^{20}} N^2$ MB. The following theorem states the result in the 1-D case.
\begin{theorem} \label{thm:memory}
Following the procedure in \Cref{thm:cod}, one only needs to sequentially construct single 1-D Loewner matrices, each of dimension $\IL_1 \in \IC^{k_l \times  k_l}$.  The largest stored matrix is $\IL_1 \in \IC^{k_{\text{max}}\times k_{\text{max}}}$, where $k_{\text{max}}=\max_{l} k_l$ ($l=1,\cdots,\ord$). In complex and double precision, the maximum disk storage is $\frac{8}{2^{20}} k_{\text{max}}^2$ \texttt{MB}.
\end{theorem}

As an illustration, for a 6-variable problem with complexity $[19,5,3,5,7,1]$, one would require $[k_1, k_2, k_3, k_4, k_5, k_6] = [20, 6, 4, 6, 8, 2]$ points, then $N=46,080$. The $\ord$-D Loewner matrix requires $31.64$ \texttt{GB} of storage, while the 1-D version would require, in the worst case scenario, i.e., for $k_{\text{max}}=20$, only $6.25$ \texttt{KB} of storage.

\begin{remark}
    In addition to computational complexity and storage, this method improves the numerical accuracy. For instance, in the modest case of a function with complexity $[9,7,2]$, the rank of the 3D Loewner matrix in floating point is much bigger than one. More details about this example are to be found in \cite{bigpaper}. The method proposed in this work, therefore, makes the computation of the barycentric weights possible.
\end{remark}

From the above considerations, it follows that the proposed null space computation method leads to a drop in not only the computational complexity of the worst-case scenario but also the memory requirements.

As illustrated in \Cref{sec:exple}, this allows treating problems with a large number of variables, in a reasonable computational time and manageable complexity, which is the main reason for claiming that the \textbf{C-o-D has been "tamed"}.

\begin{figure}
    \centering
    \includegraphics[width=10.8cm]{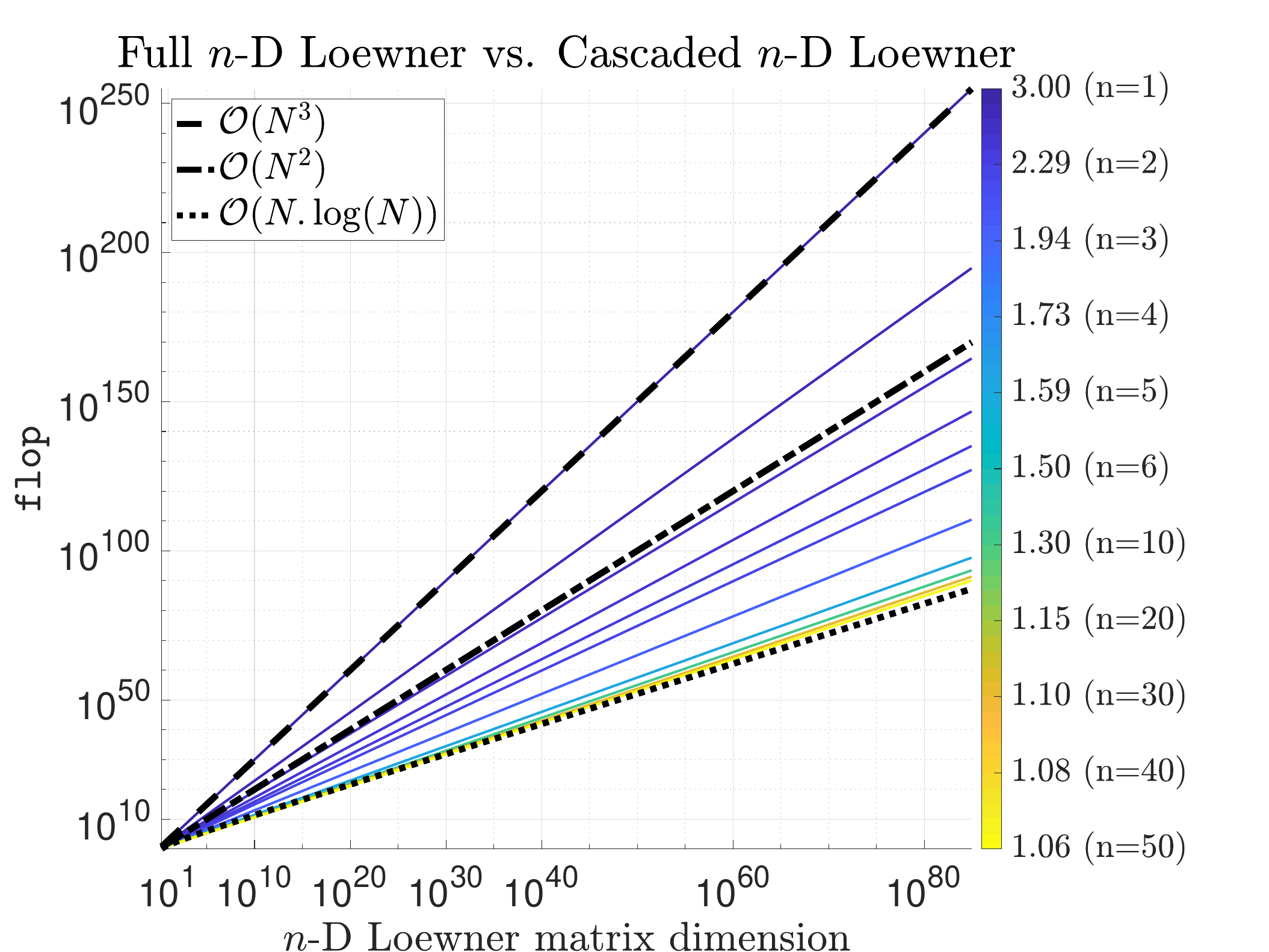}
    \caption{\flop \ comparison: cascaded $\ord$-D Loewner worst-case upper bounds for varying number of variables $\ord$, while the full $\ord$-D Loewner is ${\cal O}(N^3)$ (black dashed); comparison with ${\cal O}(N^2)$ and ${\cal O}(N \log(N))$ references are shown in dash-dotted and dotted black lines.}
    \label{fig:complexity}
\end{figure}

%% file: sec-hilbert.tex
Several researchers have contributed to sharpening Kolmogorov's original result, so currently it is often referred to as the \textit{Kolmogorov, Arnol'd, Kahane, Lorenz, and Sprecher} Theorem (see \cite{morris}, Theorem 2.1). For simplicity, we will follow \cite{morris} and state this result for $\ord=3$, in order to compare it with \Cref{thm:decoupling}.

\begin{theorem} Given a continuous function $f:\,[0,1]^3\rightarrow \mathbb{R}$ of three variables, there exist real numbers $\lambda_i$, $i=1,2$, and single-variable continuous functions $\phi_k:\,[0,1]\rightarrow\mathbb{R}$, $k=1,\cdots,7$, and a single-variable
function $g:\,\mathbb{R}\rightarrow\mathbb{R}$, such that\\[.1mm]
$\hspace*{5mm} f(x_1,x_2,x_3)=\sum_{k=1}^7\,g(\phi_k(x_1)+\lambda_1\phi_k(x_2)+\lambda_2\phi_k(x_3)),~
\forall (x_1,x_2,x_3)\in \left[0,1\right]^3$.\\[.6mm]
In the above result, $\lambda_i$ and $\phi_k$ do not depend on $f$. Thus, for $\ord=3$, eight functions are needed together with two real scalars $\lambda_i$.
\end{theorem}
The purpose of this section is to make contact with KST using a three-variable example.
\begin{example}
Consider the three-variable function   $\frH(s, t, x)=\frac{s^2+xs+1}{t+x+st+2}$. Since the degrees in each variable are $(2,1,1)$, we will need the integers
$k_1=3$, $k_2=2$, and $k_3=2$. This implies that $N=k_1k_2k_3=12$. The right and left interpolation points are chosen as $s_1=1$, $s_2=2$, $s_3=3$; $t_1=4$, $t_2=5$; $x_1=6$, $x_2=7$, and $s_4=3/2$, $s_5=5/2$, $s_6=7/2$; $t_3=9/5$, $t_4=11/5$; $x_3=13/3$, $x_4=5$, respectively. Following the theory developed above, the right triples of interpolation points are $\bS=[\bs_1,~\bs_2,~\bs_3]\otimes {\bone}_{2}\otimes{\bone}_{2}$, $\bT={\bone}_{3}\otimes[\bt_1,~\bt_2]\otimes{\bone}_{2}$, $\bX={\bone}_{3}\otimes{\bone}_{2}\otimes[\bx_1,~\bx_2]$ $\in{\IC}^{1\times N}$ (where ${\bs_i}=s-s_i$, ${\bt_i}=t-t_i$ and ${\bx_i}=x-x_i$). Thus, the resulting 3-D Loewner matrix has dimension $12\times 12$, with the $12$ barycentric weights given by (for emphasis, we denote
by \textbf{Bary} what was earlier denoted by $\bc$)
\vspace{3mm}
\begin{center}
$\textbf{Bary}=$
$\left[\frac{16}{29}~~ -\frac{17}{29}~~ -\frac{18}{29}~~ \frac{19}{29}~~ -\frac{40}{29}~~ \frac{42}{29}~~ \frac{46}{29}~~ -\frac{48}{29}~~ \frac{24}{29}~~ -\frac{25}{29}~~ -\frac{28}{29}~~ 1 \right]^\top$. \\[1mm]
\end{center}
\vspace{3mm}
As already shown, a decomposition of this vector follows, in a (point-wise) product of barycentric weights with respect to each variable, separately. Thus, \textbf{decoupling} the problem is achieved (which is one of the important aspects of KST), and the following is obtained:~ $\textbf{Bary}=\textbf{Bary}_x\odot{\textbf{ Bary}}_t\odot{\textbf{Bary}_s}$,~ where $\odot$ denotes the point-wise product. This is \cref{eq:decoupling} for $\ord=3$. This is the key result that allows the connection with KST and taming the curse of dimensionality. We have shown that the 3-D multivariate function can be computed in terms of three 1-D functions (one in each variable). These functions denoted below by $\mathbf{\Phi}(x)$, $\mathbf{\Psi}(t)$ and $\mathbf{\Omega}(s)$ are obtained from a collection of null space computations: 1 along $s$, 3 along $t$ and 6 along $x$. More specifically, following notations of \Cref{thm:decoupling}, 
\begin{equation}\label{eq:kst_decoupling}
\begin{array}{rlrl}
\bc^{x}&= \textrm{vec}\left(\begin{array}{cccccc} -\frac{16}{17} & -\frac{18}{19} & -\frac{20}{21} & -\frac{23}{24} & -\frac{24}{25} & -\frac{28}{29}\\[1mm] 1 & 1 & 1 & 1 & 1 & 1 \end{array}\right) &, \textbf{Bary}_x &= \bc^{x}\\[6mm]
\vspace{2mm}
\bc^{t}&= \textrm{vec}\left(\begin{array}{ccc} -\frac{17}{19} & -\frac{7}{8} & -\frac{25}{29}\\[1mm] 1 & 1 & 1 \end{array}\right) &, \textbf{Bary}_t &= \bc^{t}\otimes \bone_3\\[2mm]
\bc^{s}&= \textrm{vec}
\left(\begin{array}{ccc} \frac{19}{29} & -\frac{48}{29} & 1 \end{array}\right) &, \textbf{Bary}_s &= \bc^{s}\otimes \bone_{6}
\end{array}
\end{equation}
Furthermore, $\textbf{Lag}(x)$, $\textbf{Lag}(t)$ and $\textbf{Lag}(s)$ are the monomials of the \textbf{Lagrange bases components} in each variable. Finally, $\bW$ are the right interpolation values for the triples in ${\textbf S}\times{\textbf T}\times{\textbf X}$.  The ensuing numerical values are as follows:

\vspace{-3mm}

$$\hspace*{-3mm}
\underbrace{
\left[\begin{array}{c} -\frac{16}{17}\\[1mm] 1\\[1mm] -\frac{18}{19}\\[1mm] 1\\[1mm] -\frac{20}{21}\\[1mm] 1\\[1mm] -\frac{23}{24}\\[1mm] 1\\[1mm] -\frac{24}{25}\\[1mm] 1\\[1mm] -\frac{28}{29}\\[1mm] 1 \end{array}\right]}_
{\textbf{Bary}_x},
\underbrace{
\left[\begin{array}{c} -\frac{17}{19}\\[1mm] -\frac{17}{19}\\[1mm] 1\\[1mm] 1\\[1mm] -\frac{7}{8}\\[1mm] -\frac{7}{8}\\[1mm] 1\\[1mm] 1\\[1mm] -\frac{25}{29}\\[1mm] -\frac{25}{29}\\[1mm] 1\\[1mm] 1 \end{array}\right]}_
{\textbf{Bary}_t},
\underbrace{
\left[\begin{array}{c} \frac{19}{29}\\[1mm] \frac{19}{29}\\[1mm] \frac{19}{29}\\[1mm] \frac{19}{29}\\[1mm] -\frac{48}{29}\\[1mm] -\frac{48}{29}\\[1mm] -\frac{48}{29}\\[1mm] -\frac{48}{29}\\[1mm] 1\\[1mm] 1\\[1mm] 1\\[1mm] 1 \end{array}\right]}_{\textbf{Bary}_s},
\underbrace{
\left[\begin{array}{c} \frac{1}{x-6}\\[1mm] \frac{1}{x-7}\\[1mm] \frac{1}{x-6}\\[1mm] \frac{1}{x-7}\\[1mm] \frac{1}{x-6}\\[1mm] \frac{1}{x-7}\\[1mm] \frac{1}{x-6}\\[1mm] \frac{1}{x-7}\\[1mm] \frac{1}{x-6}\\[1mm] \frac{1}{x-7}\\[1mm] \frac{1}{x-6}\\[1mm] \frac{1}{x-7} \end{array}\right]}_{\textbf{Lag}(x)},
\underbrace{
\left[\begin{array}{c} \frac{1}{t-4}\\[1mm] \frac{1}{t-4}\\[1mm] \frac{1}{t-5}\\[1mm] \frac{1}{t-5}\\[1mm] \frac{1}{t-4}\\[1mm] \frac{1}{t-4}\\[1mm] \frac{1}{t-5}\\[1mm] \frac{1}{t-5}\\[1mm] \frac{1}{t-4}\\[1mm] \frac{1}{t-4}\\[1mm] \frac{1}{t-5}\\[1mm] \frac{1}{t-5} \end{array}\right]}_{\textbf{Lag}(t)},
\underbrace{
\left[\begin{array}{c} \frac{1}{s-1}\\[1mm] \frac{1}{s-1}\\[1mm] \frac{1}{s-1}\\[1mm] \frac{1}{s-1}\\[1mm] \frac{1}{s-2}\\[1mm] \frac{1}{s-2}\\[1mm] \frac{1}{s-2}\\[1mm] \frac{1}{s-2}\\[1mm] \frac{1}{s-3}\\[1mm] \frac{1}{s-3}\\[1mm] \frac{1}{s-3}\\[1mm] \frac{1}{s-3} \end{array}\right]}_{\textbf{Lag}(s)},
\underbrace{
\left[\begin{array}{c} \frac{1}{2}\\[1mm] \frac{9}{17}\\[1mm] \frac{4}{9}\\[1mm] \frac{9}{19}\\[1mm] \frac{17}{20}\\[1mm] \frac{19}{21}\\[1mm] \frac{17}{23}\\[1mm] \frac{19}{24}\\[1mm] \frac{7}{6}\\[1mm] \frac{31}{25}\\[1mm] 1\\[1mm] \frac{31}{29} \end{array}\right]}_{{\bW}}\!\!\begin{array}{c}
{\textrm def}\\
\!\!\Rightarrow
\end{array}\!\!
\left\{\!\!\begin{array}{l}
{\mathbf \Phi}(x)\!\!=  \!\!\!\textbf {Bary}_x\odot \textbf{Lag}(x)\\[2mm]
{\mathbf \Psi}(t)\!\!=  \!\!\!\textbf {Bary}_t\odot \textbf{Lag}(t)\\[2mm]
{\mathbf \Omega}(s)\!\!=\!\!\!\textbf {Bary}_s\odot \textbf{Lag}(s)\\[1mm]
\end{array}\!\!\right.
$$
With the above notation, we can express $\bH$ as the quotient of two rational functions:
$$
\left.\begin{array}{rcl}
\hat{\textbf n}(s,t,x)&=&  
\sum_{\textrm rows}
\left[{\bW}\odot{\mathbf \Phi}(x)\odot{\mathbf \Psi}(t)\odot{\mathbf \Omega}(s)\right]\\[1mm]
\hat{\textbf d}(s,t,x)&=&  \sum_{\textrm rows}
\left[{\mathbf \Phi}(x)\odot{\mathbf \Psi}(t)\odot{\mathbf \Omega}(s)\right]
\end{array}\right\}~\Rightarrow~
\frac{\hat{\textbf n}(s,t,x)}{\hat{\textbf d}(s,t,x)}={\bH}(s,t,x).
$$
Consequently, KST for rational functions, as \textbf{composition and superposition} of one-variable functions, takes the form:
\begin{equation}\label{eq:star}
\left.\begin{array}{rcl}
\hat{\textbf n}(s,t,x)&=&\sum_{\textrm rows}\,\exp \left[\,\log{\bW}+\log{\mathbf \Phi}(x)+\log{\mathbf \Psi}(t)+
\log{\mathbf \Omega}(s)\,\right]\\[1mm]
\hat{\textbf d}(s,t,x)&=&\sum_{\textrm rows}\,\exp \left[\,\log{\mathbf \Phi}(x)+\log{\mathbf \Psi}(t)+
\log{\mathbf \Omega}(s)\,\right].
\end{array}\right\}
\end{equation}

\end{example}

\noindent
\textbf{Similarities and differences between KST and the results in \cref{eq:decoupling} and \cref{eq:star}}

\vspace*{2mm}
\noindent
\textbf{a).} While KST refers to continuous functions defined on $[0,1]^\ord$, \cref{eq:star} is concerned with rational functions defined on $\mathbb{C}^\ord$.\\[1mm]
\textbf {b).} Expressions in \cref{eq:star} are valid in a particular basis, namely the \textit{Lagrange} basis. Multiplication of functions in \cref{eq:star},  is defined with respect to this basis.\\[1mm]
\textbf {c).} The composition and superposition properties hold for the numerator and denominator. This is important in our case because \cref{eq:star} preserves interpolation conditions.\\[1mm]
\textbf {d).} The parameters needed are $\ord=3$ Lagrange bases (one in each variable) and the barycentric coefficients of the numerator and denominator. Note that in KST, no explicit denominators are considered.\\[1mm]
\textbf {e).}  Both KST and \cref{eq:star} accomplish the goal of replacing the computation of multivariate functions by means of a series of computations involving single-variable functions, KST for general continuous functions, and \cref{eq:star} for rational functions. Notice also that \cref{eq:star} provides a different formulation of the problem than KST.\\[1mm]
\textbf {f).} In addition to the Kolmogorov-Arnold neural nets (KANs) \cite{KAN}, our approach provides a new application of KST to the modeling of multi-parameter systems.

%% file: sec-algo.tex
This section focuses on the numerical aspects of constructing the realization from data measurements.

\subsection{Two algorithms}

In what follows, we introduce two algorithms. The first one is a direct method extending the one proposed by the authors in \cite{IA2014}, while the second is an iterative one,  inspired by the p-AAA presented in \cite{CBG2023}. These procedures are outlined in Algorithm \ref{algo:LL_nD} and in Algorithm \ref{algo:LL_nD_adaptive}. 
For additional details, see also \cite{IA2014,CBG2023}.

\begin{algorithm}
\begin{algorithmic}[1]
\REQUIRE $\tableau{\ord}$ as in \Cref{tab:nD}
\STATE Check that interpolation points are disjoint. 
\STATE Compute $d_l=\max_k \rank \IL_1^{(k)}$, the order along variable $\var{l}$ ($k$ is referring to all possible \\ combinations for the frozen variables $\{\var{1},\cdots,\var{k-1},\var{k+1},\cdots,\var{\ord}\}$).
\STATE Construct \cref{eq:data_n}, a sub-selection $P_c^{(\ord)}$ where $(k_1,k_2\dots,k_\ord)=(d_1,d_2,\dots,d_\ord)+1$; and $P_r^{(\ord)}$ where $(q_1,q_2\dots,q_\ord)$ gather the rest of the data. 
\STATE Compute $\bc_\ord$, the $\ord$-D Loewner matrix null space  e.g. using \Cref{thm:cod}.
\STATE Construct $\IAlag$, $\IBlag$, $\bGamma$ and $\bDelta$ as in \Cref{res:realn} with any left/right separation.
\STATE Construct multivariate realization as in  \Cref{thm:realization}.
\ENSURE $\bH(\var{1},\dots,\var{\ord}) = \bC \bPhi(\var{1},\var{2},\dots,\var{\ord})^{-1}\bG$ interpolates $\frH(\var{1},\var{2},\dots,\var{\ord})$ along $P_c^{(\ord)}$.
\end{algorithmic}
\vspace{4mm}
\caption{Direct data-driven pROM construction}
 \label{algo:LL_nD} 
\end{algorithm}

\begin{algorithm}
\begin{algorithmic}[1]
\REQUIRE $\tableau{\ord}$ as in \Cref{tab:nD} and tolerance $\texttt{tol}>0$
\STATE Check that interpolation points are disjoint. 
\WHILE{$\texttt{error} > \texttt{tol}$} 
\STATE Search the point indexes with maximal $\texttt{error}$ (first iteration: pick any set).
\STATE Add points in $P_c^{(\ord)}$ and put the remaining ones in $P_r^{(\ord)}$, obtain \cref{eq:data_n}.
\STATE Compute $\bc_\ord$, the $\ord$-D Loewner matrix null space  e.g. using \Cref{thm:cod}.
\STATE Construct $\IAlag$, $\IBlag$, $\bGamma$ and $\bDelta$ as in \Cref{res:realn} with any left/right separation.
\STATE Construct multivariate realization as in  \Cref{thm:realization}.
\STATE Evaluate $\texttt{error} = \max ||\widehat{\tableau{\ord}}-\tableau{\ord}||$ where $\widehat{\tableau{\ord}}$ is the evaluation  of $ \bH(\var{1},\dots,\var{\ord})$ \\ along the support points.
\ENDWHILE
\ENSURE $\bH(\var{1},\dots,\var{\ord}) = \bC \bPhi(\var{1},\var{2},\dots,\var{\ord})^{-1}\bG$ interpolates $\frH(\var{1},\var{2},\dots,\var{\ord})$ along $P_c^{(\ord)}$.
\end{algorithmic}
\vspace{4mm}
\caption{Adaptive data-driven pROM construction} 
\label{algo:LL_nD_adaptive} 
\end{algorithm}

\subsection{Discussion} 

The main difference between the two algorithms is that Algorithm \ref{algo:LL_nD} is direct while Algorithm \ref{algo:LL_nD_adaptive} is iterative. Indeed, in the former case, the order is estimated at step 2, while the order is iteratively increased in the latter case until a given accuracy is reached. 

By analyzing Algorithm \ref{algo:LL_nD}, the process first needs to estimate the rational order along each variable $\var{l}$. Then, we construct the interpolation set \cref{eq:data_n} (here, one may shuffle data and interpolate different blocks). From this initial data set, the $\ord$-D Loewner matrix and its null space may be computed using either the full (\Cref{sec:LL}) or the 1-D recursive (\Cref{sec:cod}) approach. Based on the barycentric weights, the realization is constructed using \Cref{res:realn}.

%By now considering \Cref{algo:LL_nD_adaptive}, the process is almost similar. 
The difference between the two algorithms consists of the absence of the order detection process
in the second algorithm. Instead, it is replaced by an evaluation of the model along the data set at each step until a tolerance is reached. Then, at each iteration, one adds the support points set where the maximal error between the model and the data occurs. This idea is originally exploited in the univariate case of AAA in \cite{nakatsukasa2018} and its parametric version from \cite{CBG2023}; we similarly follow this greedy approach. 

\subsubsection{Dealing with real arithmetic}

All computational steps have been presented using complex data. 
%This is the most straightforward manner to present this generic method. 
However, in applications, it is often desirable to deal with real-valued functions in order to preserve the realness of the realization and to allow the time-domain simulations of the differential-algebraic equations. To do so, some assumptions and adaptations must be satisfied. Basically, interpolation points along each variable must be either real or chosen to be closed under conjugation. For more details on this procedure, we refer the reader to \cite[Section A.2]{IA2014}.

\subsubsection{Null space computation remarks}

To apply the proposed methods to a broad range of real-life applications, we want to comment on the major computational effort / hard point in the proposed process: the \textbf{null space computation}. Indeed, either in the full $\ord$-D and the recursive $1$-D case, a null space must be computed. Numerically, there exist multiple ways to compute it: \texttt{SVD} or \texttt{QR} decomposition, linear resolution, etc. Without going into details outside the scope of this paper, many tuning variables may be adjusted to improve accuracy. These elements are crucial to the success of the proposed solution. %Investigations are left (and not forgiven) to future works.
In the next section, all null spaces have been computed using the standard \texttt{SVD} routine of \texttt{MATLAB}. For more details, the reader may refer to \cite{guglielmi2015}.

%% file: sec-examples.tex
The effectiveness of the numerical procedures sketched in Algorithm \ref{algo:LL_nD} and in Algorithm \ref{algo:LL_nD_adaptive} is illustrated in this section, through examples involving multiple variables ranging from two to twenty.
%\footnote{We also provide numerical and open-access examples at \url{https://sites.google.com/site/charlespoussotvassal/nd_loew_tcod}.}. 
In what follows, the computations have been performed on an Apple MacBook Air with 512 \texttt{GB SSD} and 16 \texttt{GB RAM}, with an \texttt{M1} processor. The software used was \texttt{MATLAB} 2023b.

%%%%%%%%%%%%%%%%%%%%%%%%%%%%%%%%%%%%%%
\subsection{A simple synthetic parametric model (2-D)}

Let us start with the simple example used in \cite[Section 5.1]{IA2014} and \cite[Section 3.2.1]{CBG2023},  for which the target function is
$\frH(s,p) = \frac{1}{1+25(s+p)^2}+\frac{0.5}{1+25(s-0.5)^2}+\frac{0.1}{p+25}$.

We use the same sampling setting as in the references mentioned above. Along the $s$ variable, we choose 21 points linearly spaced in $[-1,1]$. For the direct method in Algorithm \ref{algo:LL_nD}, we alternatively sample the grid as  $\lan{1}=[-1,-0.8,\cdots,1]$ and $\mun{1}=[-0.9,-0.7,\cdots,0.9]$; then, along the $p$ variable, we choose $21$ linearly spaced points in $[0,1]$. For the direct approach in Algorithm \ref{algo:LL_nD}, we alternatively sample the grid as $\lan{2}=[0,0.1,\cdots,1]$ and $\mun{2}=[0.05,0.15,\cdots,0.95]$. 

First, we apply Algorithm \ref{algo:LL_nD} and obtain the single variables singular value decay reported in \Cref{fig:2simple_synthetic_sv} (left), suggesting approximation orders along $(s,p)$ of $(d_1,d_2)=(4,3)$, being precisely the one of the equation $\frH(s,p)$ above. Then, the $2$-D Loewner matrix is constructed and its associated singular values are reported in \Cref{fig:2simple_synthetic_sv} (right), leading to the full null space and barycentric weights (results follow next).

\begin{figure}[ht!]
    \centering
        \hspace{-4mm}
    \includegraphics[width=0.53\textwidth]{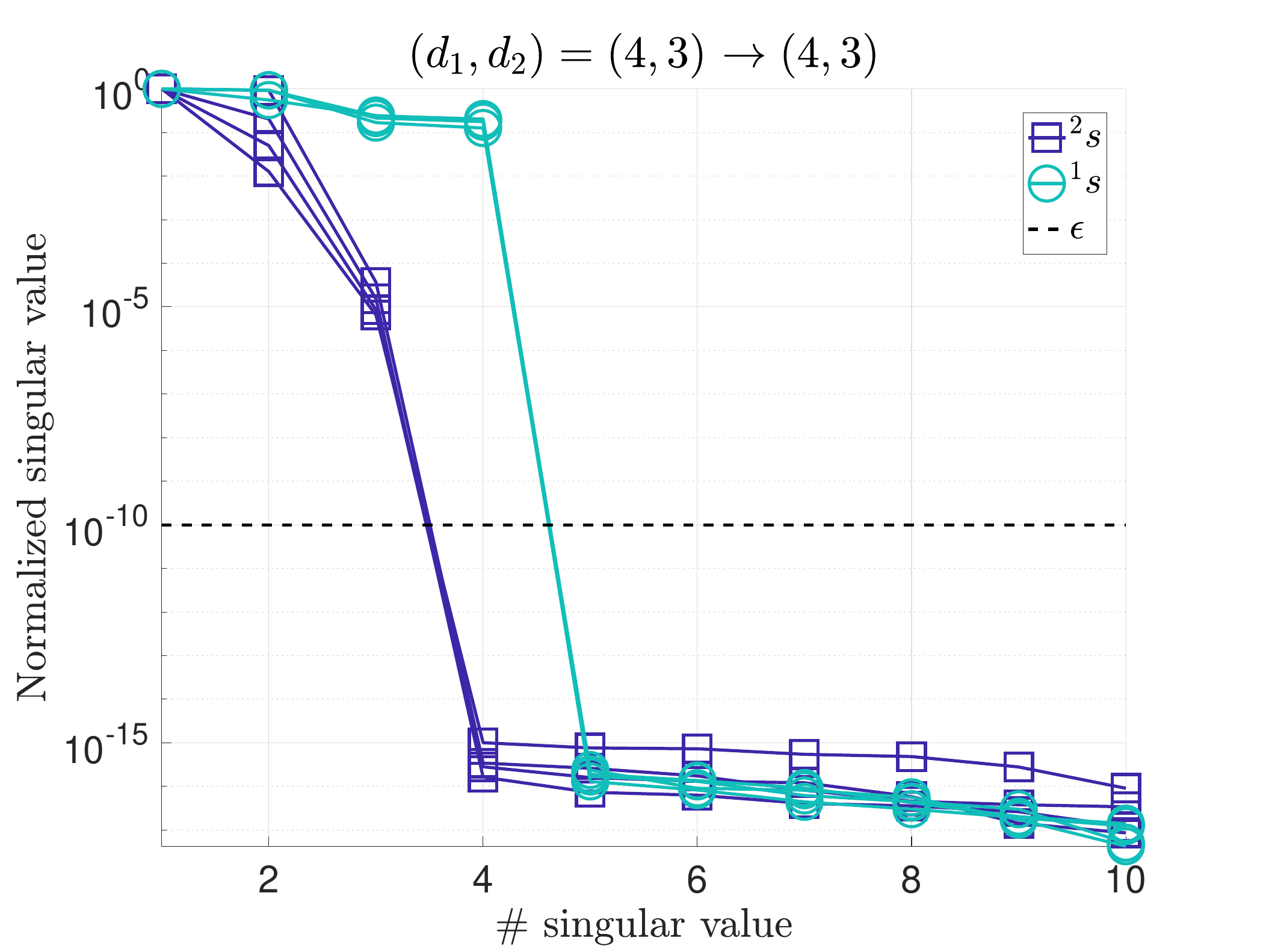}
    \hspace{-10mm}
    \includegraphics[width=0.53\textwidth]{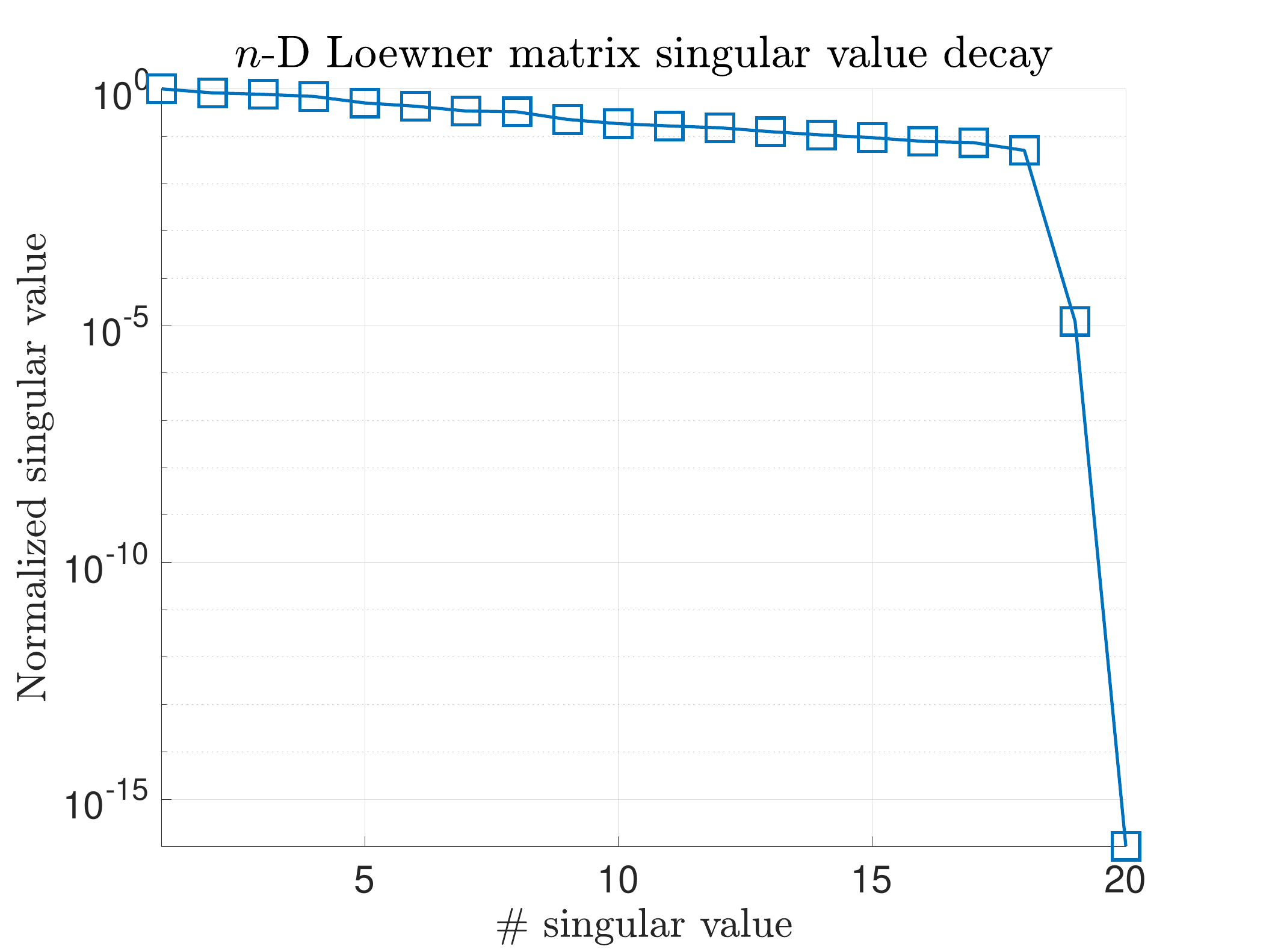}
    \caption{2-D simple synthetic model: Algorithm \ref{algo:LL_nD} normalized singular values of each 1-D (left) and the 2-D (right) Loewner matrices.}  \label{fig:2simple_synthetic_sv}
\end{figure}

Next, we investigate the behavior of Algorithm \ref{algo:LL_nD_adaptive}. In \Cref{tab:2D_synthetic_simple}, we report the iterations of this algorithm when computing the null space with either the full 2-D version (\Cref{tab:2D_synthetic_simple_1}) or the recursive 1-D one (\Cref{tab:2D_synthetic_simple_2}). In both cases, the same order is recovered, i.e., $(4,3)$. Even if the selected interpolation points are slightly different, the final error is below the chosen tolerance, which is \texttt{tol=$10^{-6}$}. By now comparing the \flop \ complexity, the benefit of the proposed recursive 1-D approach with respect to the 2-D one is clearly emphasized, even for such a simple setup. Indeed, while the latter is of $=1+2^3+6^3+12^3+20^3=9,953$ \flop, the former leads to $2+10+51+172+445=680$ \flop, being $14$ times smaller. The mismatch for the three configurations over all sampling points of $\tableau{2}$ data is close to machine precision for all configurations. 

\begin{table}[ht!]
    \begin{subtable}[b]{.5\linewidth}
    \centering
\small
    \subcaption{Algorithm \ref{algo:LL_nD_adaptive} (full $\IL_\ord$)}
    \small
    \vspace{1mm}
    \begin{tabular}{|c|c|c|c|c|}
    \hline
    \footnotesize
     Iter. & $\lan{1} $ & $\lan{2}$ & $(k_1,k_2)$ & \flop \\ \hline
     1 &  $0$& $0$ & $(1,1)$ & $1^3$ \\ \hline
     2 &  $-1$&  & $(2,1)$ & $2^3$ \\ \hline
     3 &  $-0.9$& $0.9$  & $(3,2)$ & $6^3$\\ \hline
     4 &  $-0.1$& $0.2$ & $(4,3)$ & $12^3$\\ \hline
     5 & $0.6$ &  $1$& $(5,4)$ & $20^3$\\ \hline
    \end{tabular}
    \label{tab:2D_synthetic_simple_1}
    \end{subtable}
    \begin{subtable}[b]{.4\linewidth}
    \centering
    \subcaption{Algorithm \ref{algo:LL_nD_adaptive} (recursive $\IL_1$)}
        \vspace{1mm}
    \small
    \begin{tabular}{|c|c|c|c|}
    \hline
     $\lan{1} $ & $\lan{2}$ & $(k_1,k_2)$ & \flop \\ \hline
     $0$ & $0$ & $(1,1)$ & 2 \\ \hline
     $-1$ &  & $(2,1)$ & 10 \\ \hline
     $0.1$ & $0.05$ & $(3,2)$ & 51 \\ \hline
     $-0.9$ & $0.75$ & $(4,3)$ & 172 \\ \hline
     $0.7$ & $0.15$ & $(5,4)$ & 445 \\ \hline
    \end{tabular}
    \label{tab:2D_synthetic_simple_2}
    \end{subtable}
    \vspace{-2mm}
\caption{2-D simple model iterations with different null space computation methods.}\label{tab:2D_synthetic_simple}
\end{table}
\normalsize

Finally, to conclude this first example, \Cref{fig:2simple_synthetic_fr} reports the responses (left) and mismatch (right) along $s$ for different values of $p$, for the original model and the ones obtained with Algorithm \ref{algo:LL_nD} and with Algorithm \ref{algo:LL_nD_adaptive} (with recursive 1-D null space). 

\begin{figure}[h!]
    \hspace{-6mm}
    \centering   \includegraphics[width=0.53\textwidth]{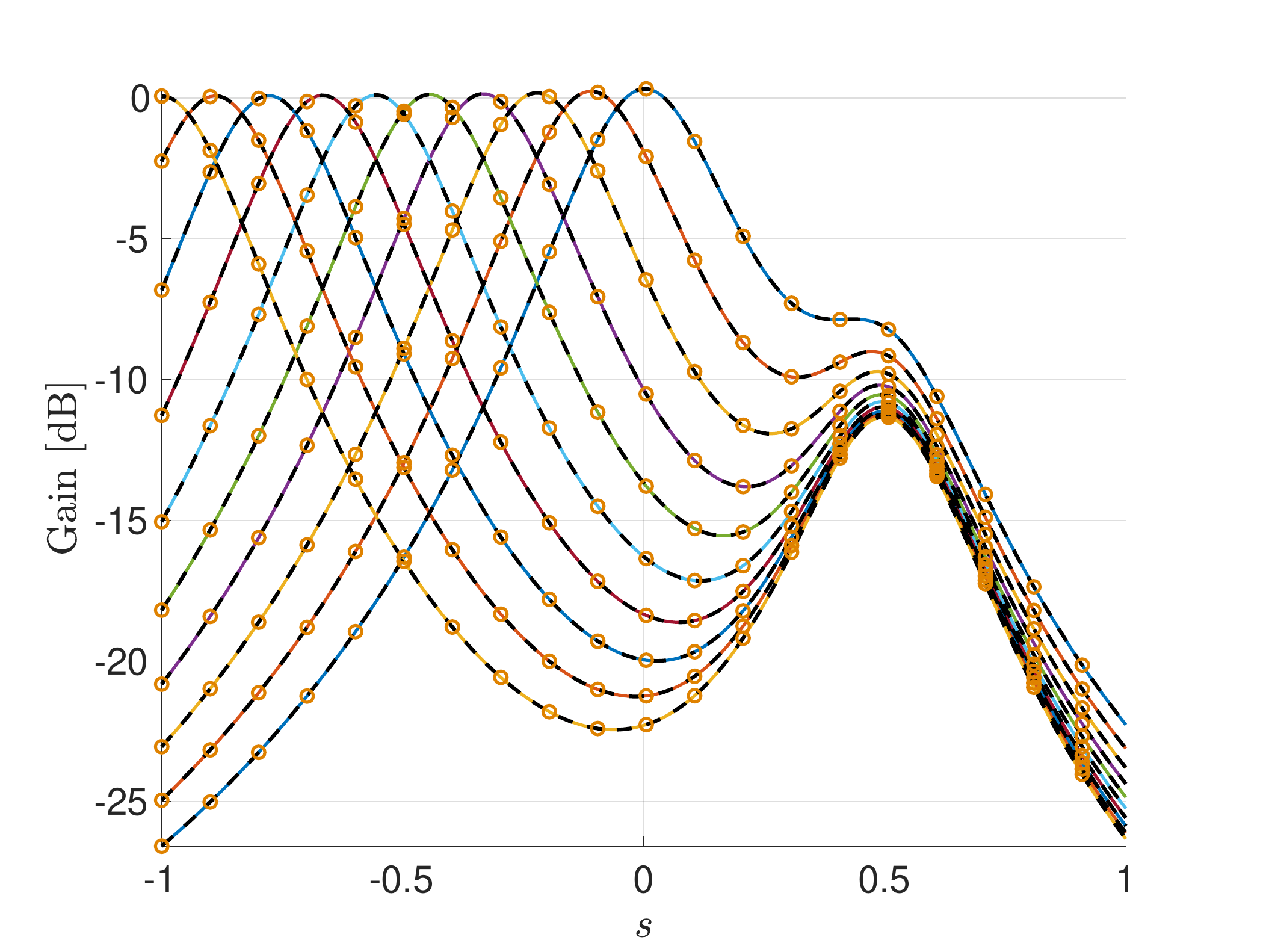}    
    \hspace{-8mm}
    \includegraphics[width=0.53\textwidth]{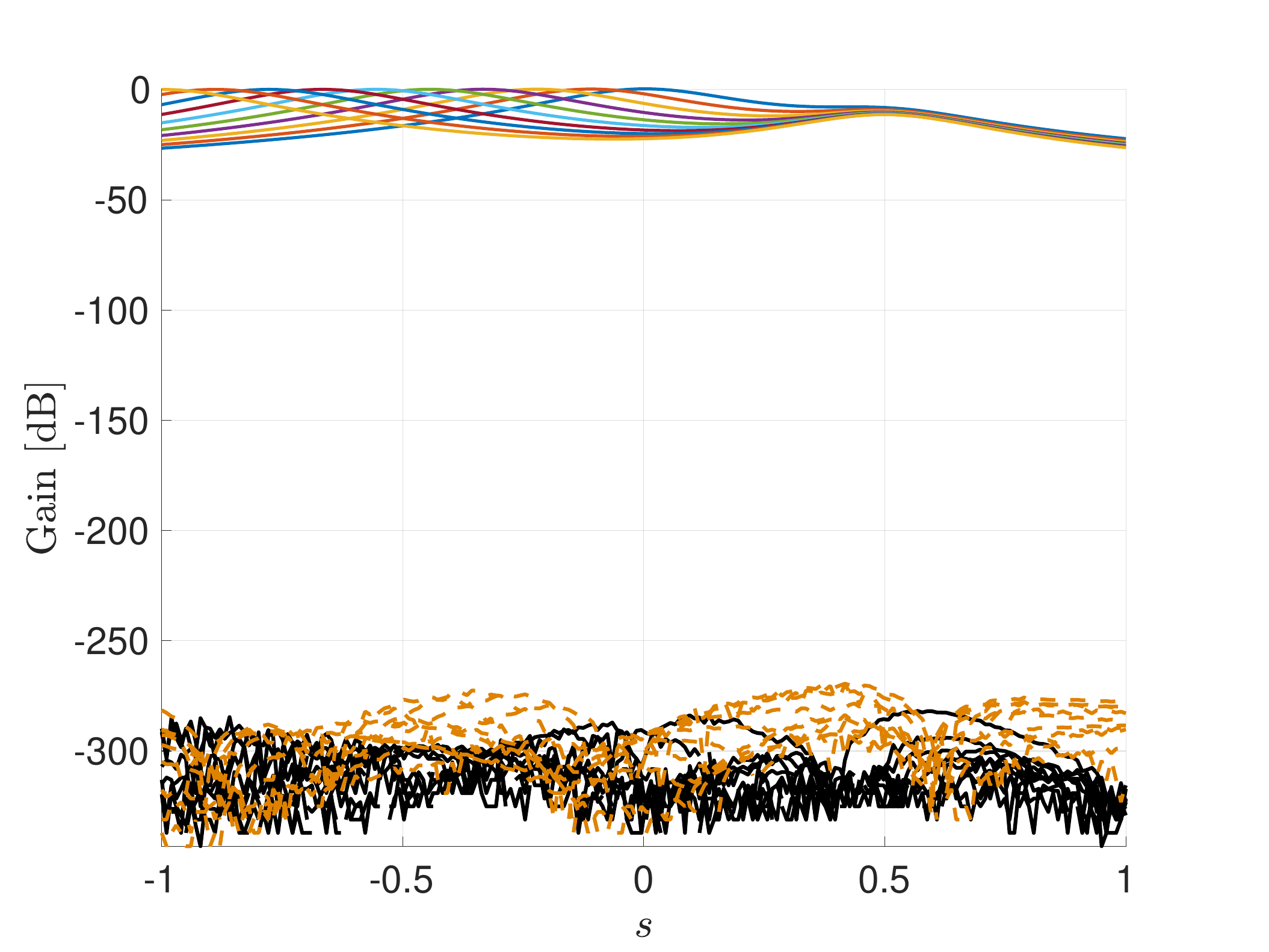}
    \caption{2-D simple model: frequency responses (left) and errors (right); original compared to Algorithm \ref{algo:LL_nD} (black lines), and to Algorithm \ref{algo:LL_nD_adaptive} (orange dots and dashed lines).}
    \label{fig:2simple_synthetic_fr}
\end{figure}

%%%%%%%%%%%%%%%%%%%%%%%%%%%%%%%%%%%%%%
\subsection{Flutter phenomena for flexible aircraft (3-D)}

This numerical example is extracted from industrial data and considers a mixed model/data configuration. It represents the flutter phenomena for flexible aircraft as detailed in \cite{RPVT2023}\footnote{We acknowledge P. Vuillemin for generating the (modified) data.}. 
This model can be described as $s^2M(m) x(s) + s B(m) x(s) + K(m) x(s) - G(s,v) = u(s)$, where $M(m), B(m), K(m)\in\IR^{n\times n}$ are the mass, damping, and stiffness matrices, all dependent on the aircraft mass $m\in \IR_+$ ($n\approx 100$). These matrices are constant for a given flight point (but vary for a mass configuration). Then, the generalized aeroelastic forces $G(s,v)\in\IC^{n\times n}$ describe the aeroelastic forces exciting the structural dynamics. This $G(s,v)$ is known only at a few sampled frequencies and some true airspeed, \emph{i.e.}, $G(\imath \omega_i,v_j)$ where $i=1,\cdots,150$ and $j=1,\cdots,10$. Note that these values are obtained through dedicated high-fidelity numerical solvers. The sampling setup is as follows. Along the $s$ variable, $\lan{1}$ are 150 logarithmically spaced points between $\imath [10,35]$ and $\mun{1}=-\lan{1}$; Along the $v$ variable, $\lan{2}$ are 5 linearly spaced points between $[4.77,5.21]\cdot 10^3$ and $\mun{2}$, 5 linearly spaced between $[4.82,5.27]\cdot 10^3$; Along the $m$ variable, $\lan{3}$ are 5 linearly spaced points between $[1.52,1.66]\cdot 10^3$ and $\mun{2}$, 5 linearly spaced between $[1.54,1.68]\cdot 10^3$.

Here, the data is a $3$-dimensional tensor $\tableau{3}\in \IC^{300\times 10\times 10}$. By applying Algorithm \ref{algo:LL_nD}, an approximation order $(14,1,1)$ is reasonable. The singular value decay of the 3-D Loewner matrix is reported in \Cref{fig:3flutter_fr} (left). Then, the original frequency response and that of the pROM are displayed in \Cref{fig:3flutter_fr} (right), showing an accurate matching attained.

\begin{figure}[ht!]
    \centering
        \hspace{-7mm}
    \includegraphics[width=0.545\textwidth]{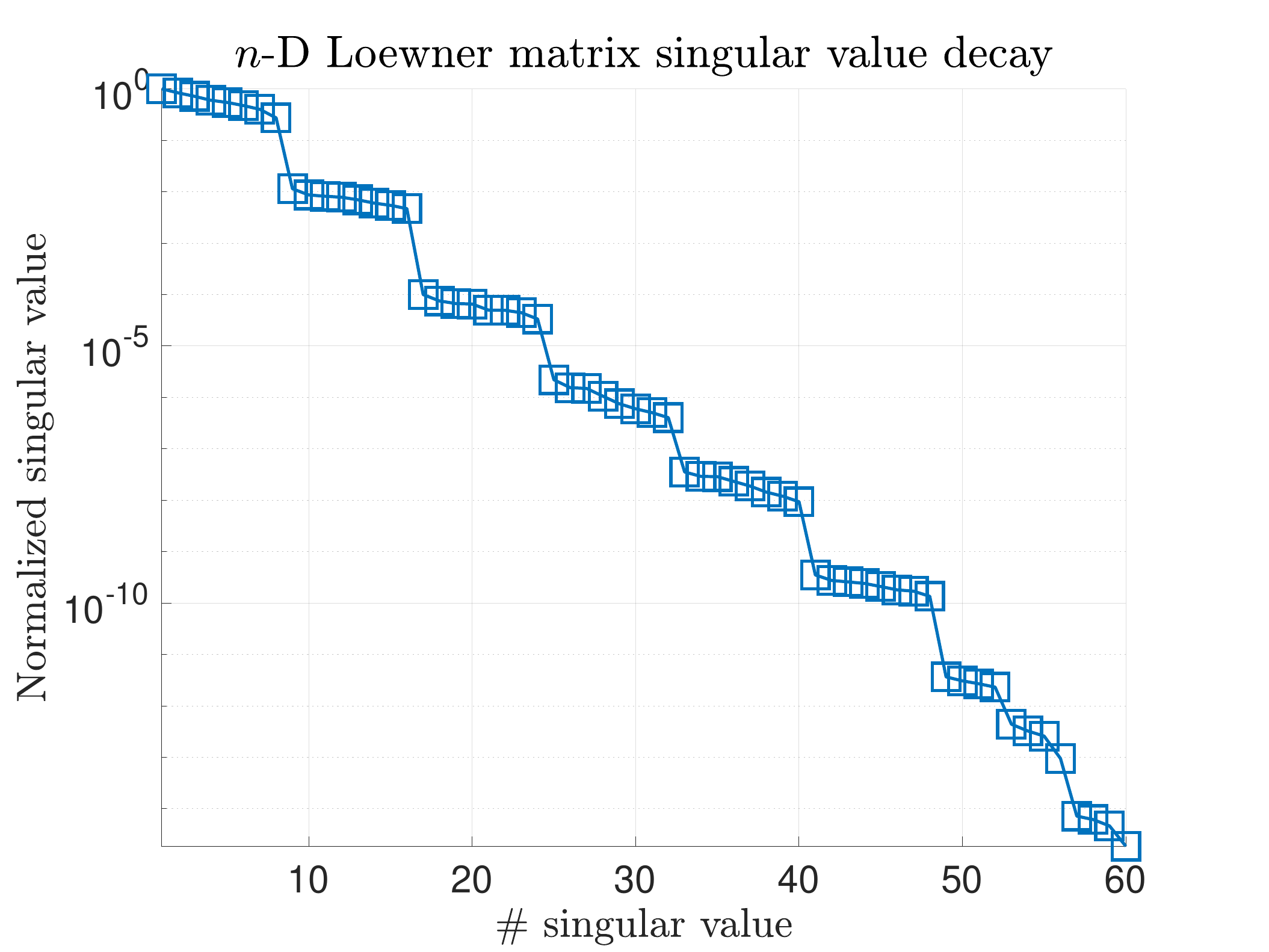}
    \hspace{-10mm}
    \includegraphics[width=0.545\textwidth]{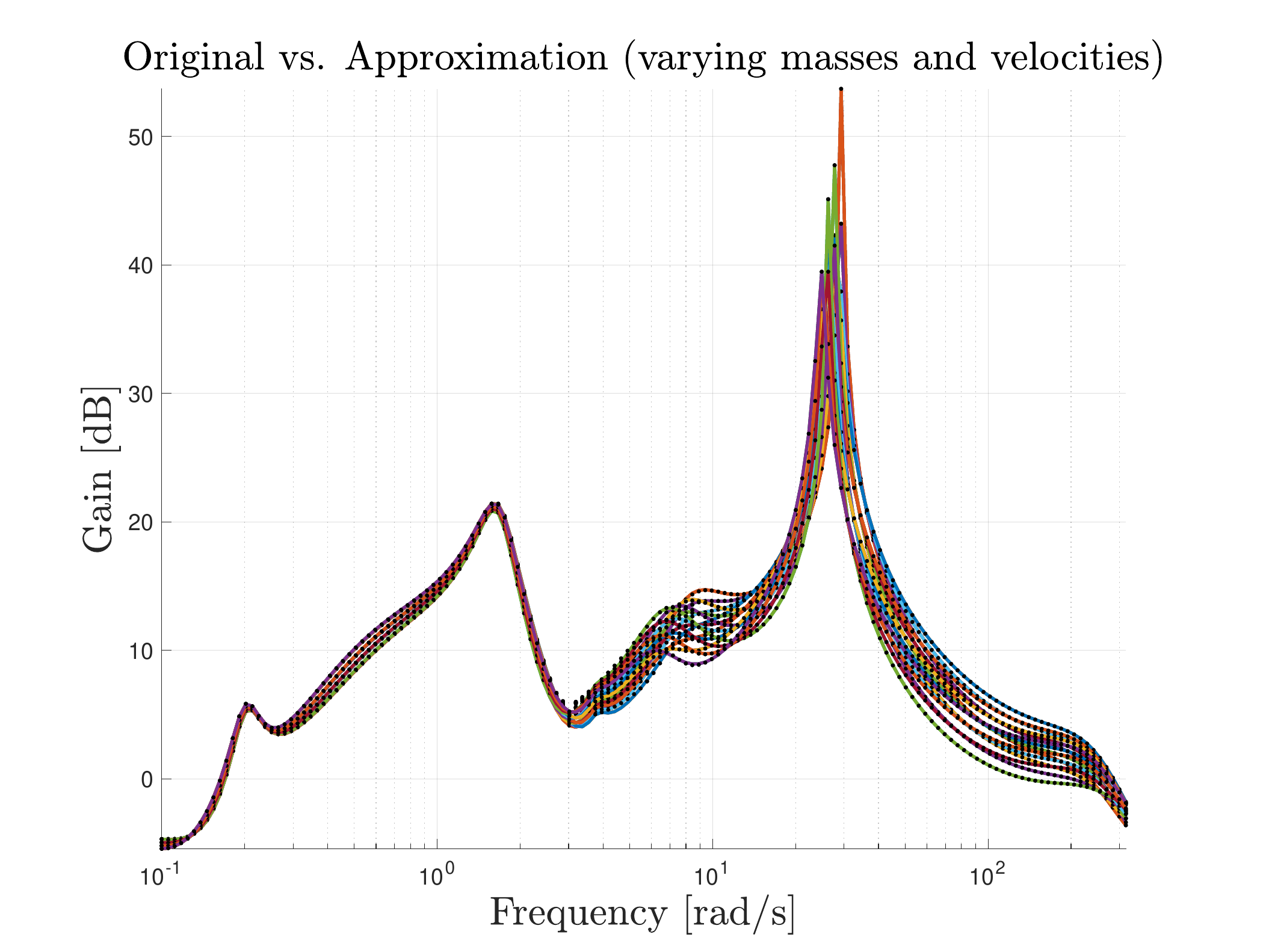}
    \vspace{-3mm}
    \caption{3-D flutter model: 3-D Loewner matrix singular values (left) and frequency responses (right). Original (solid colored) and pROM (black dotted).}
    \label{fig:3flutter_fr}
\end{figure}

One relevant point of the proposed Loewner framework, nicely illustrated in this application, is its ability to construct a realization of a pROM, based on a hybrid data set, mixing frequency-domain data and matrices. By connecting this problem to NEPs, parametric rational approximation allows estimating the eigenvalue trajectories; we refer to \cite{QVP2021,VQPV2023,RPVT2023} for details and industrial applications.

%%%%%%%%%%%%%%%%%%%%%%%%%%%%%%%%%%%%%%
\subsection{A multivariate function with a high number of variables (20-D)}

To conclude and to numerically demonstrate the scalability features of our process, let us consider the following 20-variable rational model $\frH(\var{1},\ldots,\var{20})=$
\begin{equation*}%\label{eq:20D} 
\dfrac{3\cdot\var{1}^3+4\cdot\var{8}+\var{12}+\var{13}\cdot\var{14}+\var{15}}{\var{1}^{10}+\var{2}^2\cdot\var{3}+\var{4}+\var{5}+\var{6}+\var{7}\cdot\var{8}+\var{9}\cdot\var{10}\cdot\var{11}+\var{13}+\var{13}^3\cdot\pi+\var{17}+\var{18}\cdot\var{19}-\var{20}},
\end{equation*}
\normalsize
with a complexity of $(10, 2, 1, 1, 1, 1, 1, 1, 1, 1, 1, 1, 1, 1, 1, 3, 1, 1, 1, 1)$. By applying Algorithm \ref{algo:LL_nD} with the recursive 1-D null space construction, the barycentric coefficients $\bc_\ord\in\IC^{17301504}$ are obtained with a computational complexity of 149,226,836 \flop, computed in 4 hours. As explained in the supplementary material, this vector allows reconstructing the original model with an absolute error $\approx 10^{-7}$, for a random parameter selection. Applying the full $\ord$-D Loewner version instead would have theoretically required the construction of a Loewner matrix of dimension $N=17,301,504$, with a null space computation costing about $5.18 \cdot 10^{21}$ \flop,  being prohibitive on a standard computer. Storing such a $N\times N$ $\ord$-D Loewner matrix would require  $4,356$ \texttt{TB} in double precision, while the 1-D approach needs $1.89$ \texttt{KB} only (in the worst case).